\newif\ifanonymized
\ifdefined\ANONYMIZED
  \anonymizedtrue
\else
  \anonymizedfalse
\fi

\documentclass[a4paper,fleqn]{cas-sc}

\usepackage[numbers]{natbib}

\usepackage{amsthm}

\def\tsc#1{\csdef{#1}{\textsc{\lowercase{#1}}\xspace}}
\tsc{WGM}
\tsc{QE}
\usepackage{amsmath, mathrsfs, amsfonts, mathtools, amssymb, amscd, ascmac}
\usepackage{bm}
\usepackage{csvsimple}
\usepackage{siunitx}

\usepackage{algorithm,algorithmic}

\usepackage[caption=false]{subfig}
\usepackage{float}
\usepackage{enumitem}
\setlist[enumerate]{leftmargin=.5in}
\setlist[itemize]{leftmargin=.5in}

\newcommand{\A}{\mathcal{A}}
\newcommand{\B}{\mathcal{B}}
\newcommand{\E}{\mathbb{E}}

\renewcommand{\L}{\mathcal{L}}
\newcommand{\N}{\mathbb{N}}
\newcommand{\Prob}{\mathbb{P}}
\newcommand{\Proj}{\Pi}
\newcommand{\R}{\mathbb{R}}
\newcommand{\Sb}{\mathbb{S}}

\newcommand{\bzero}{\bm{0}}
\newcommand{\bone}{\bm{1}}
\newcommand{\bw}{\bm{w}}
\newcommand{\bdelta}{\bm{\delta}}
\newcommand{\bepsilon}{\bm{\epsilon}}
\newcommand{\bphi}{\bm{\phi}}
\newcommand{\bxi}{\bm{\xi}}

\newcommand{\norm}[1]{|#1|}
\newcommand{\pnorm}[1]{\|#1\|}
\newcommand{\opnorm}[1]{\left|#1\right|_{\L}}
\newcommand{\opnormsub}[2]{|#1|_{\L(#2)}}
\newcommand{\fnorm}[1]{\left|#1\right|_{F}}
\newcommand{\wc}{{}\cdot{}}
\newcommand{\bracket}[2]{\left\langle #1, #2 \right\rangle}

\DeclareMathOperator{\ran}{ran}
\DeclareMathOperator{\cov}{Cov_{\mathrm{ens}}}

\newcommand{\x}{u}
\newcommand{\bx}{\bm{u}}
\newcommand{\Xe}{U}

\newcommand{\z}{v}
\newcommand{\bz}{\bm{\z}}
\newcommand{\Ze}{V}
\newcommand{\Zehat}{\widehat{\Ze}}
\newcommand{\zf}{\widehat{\z}}
\newcommand{\bzf}{\widehat{\bz}}

\newcommand{\Pf}{\widehat{P}}
\newcommand{\Pa}{P}
\newcommand{\za}{\z}

\newcommand{\y}{y}
\newcommand{\by}{\bm{\y}}

\newcommand{\Ny}{N_y}
\newcommand{\Ne}{m}

\newcommand{\Mdyn}{\Psi}
\newcommand{\Hobs}{H}
\newcommand{\Robs}{R}
\newcommand{\robs}{r}

\newcommand{\berr}{\bdelta}
\newcommand{\berrhat}{\widehat{\bdelta}}

\definecolor{RevisionBlue}{RGB}{0,0,160}
\newcommand{\rev}[1]{#1}
\newif\ifinrevision
\newenvironment{revision}
  {\begingroup}
  {\endgroup}
\makeatletter
\apptocmd{\@floatboxreset}{\ifinrevision\color{RevisionBlue}\fi}{}{}
\makeatother

\newtheorem{thm}{Theorem}[section] 
\newtheorem{theorem}[thm]{Theorem}
\newtheorem{lemma}[thm]{Lemma}
\newtheorem{corollary}[thm]{Corollary}

\newdefinition{remark}{Remark}[section] 
\newdefinition{definition}{Definition}[section] 

\begin{document}
\let\WriteBookmarks\relax
\def\floatpagepagefraction{1}
\def\textpagefraction{.001}

\shorttitle{Information transfer in filtering for the partially observed Lorenz 96 model}    

\ifanonymized
  \shortauthors{Anonymous}
\else
  \shortauthors{K. Takeda}
\fi

\title [mode = title]{Observed-unobserved information transfer in ensemble Kalman filtering for the partially observed Lorenz 96 model}  

\ifanonymized
\else
  \tnotemark[1,2]
\fi

\tnotetext[2]{Author accepted manuscript, accepted for publication in \textit{Nonlinear Science}. \copyright\ 2026. This manuscript version is made available under the CC BY-NC-ND 4.0 license: \url{https://creativecommons.org/licenses/by-nc-nd/4.0/}.}
\ifanonymized
\else
  \tnotetext[1]{This work was supported by the RIKEN Junior Research Associate Program, JST SPRING Grant Number JPMJSP2110, and \rev{JSPS KAKENHI Grant Number JP26K17033}.}
\fi

%

\ifanonymized
  \author[1]{Anonymous}
\else
  \author[1]{Kota Takeda}[orcid=0000-0001-6789-7842]
\fi

\cormark[1]


\ifanonymized
  \ead{}
\else
  \ead{takeda@na.nuap.nagoya-u.ac.jp}
\fi


\credit{Conceptualization, Methodology, Software, Formal analysis, Investigation, Writing, Visualization, Funding acquisition}

\ifanonymized
  \affiliation[1]{}
\else
  \affiliation[1]{organization={Department of Applied Physics, Graduate School of Engineering, Nagoya University},
              addressline={Furo-cho, Chikusa-ku}, 
              city={Nagoya},
              postcode={464--8603}, 
              state={Aichi},
              country={Japan}}
\fi







\cortext[1]{Corresponding author}



\begin{abstract}
We consider filtering problems with partial observations of nonlinear dynamical systems, where accurate state estimation relies on information transfer from observed to unobserved components.
Although the ensemble Kalman filter (EnKF) is widely used in this setting, rigorous error estimates remain limited because the information transfer mechanism induces non-symmetric matrix structures that are difficult to analyze.
Existing analyses often avoid this difficulty by \rev{restricting the covariance or inflation to the observed subspace, which suppresses} the observed-unobserved coupling.
In this study, we \rev{establish uniform-in-time error bounds} for the EnKF applied to the partially observed Lorenz 96 model, a prototypical chaotic dynamical system\rev{, both with and without the covariance projection}.
Numerical experiments \rev{are consistent with} the theoretical findings and \rev{show how the covariance projection affects the observed--unobserved information flow and resulting accuracy under comparatively sparse observations}.
\end{abstract}




\maketitle

\section{Introduction}
\label{sec:introduction}
We study the filtering problem for chaotic dynamical systems.
In this problem, small initial errors rapidly grow due to the chaotic nature of the systems. To suppress this error growth, noisy and partial observations are incorporated in data assimilation approaches.
We focus on filters based on the least squares method, or equivalently the Gaussian approximation to compute the Bayesian posterior distributions.
The three-dimensional variational method (3DVar)~\cite{lorencAnalysisMethodsNumerical1986} uses a fixed-in-time background covariance matrix $\Pf$, which provides uncertainty in the model predictions as prior information.
On the other hand, the ensemble Kalman filter (EnKF)~\cite{burgersAnalysisSchemeEnsemble1998,andersonEnsembleAdjustmentKalman2001,bishopAdaptiveSamplingEnsemble2001,evensenDataAssimilationEnsemble2009} computes the covariance $\Pf_n$ dynamically using an ensemble of predictions. Although 3DVar is computationally cheap, it does not capture flow-dependent uncertainty. In contrast, the EnKF adapts to time-varying uncertainty arising from nonlinear dynamics.

Mathematical analysis investigates the fundamental properties of these filters\rev{, including well-definedness, consistency, sampling errors, filter stability~\cite{lawDataAssimilationMathematical2015,tongNonlinearStabilityEnsemble2016,mandelConvergenceEnsembleKalman2011,kwiatkowskiConvergenceSquareRoot2015,al-ghattasNonasymptoticAnalysisEnsemble2024,al-ghattasCovarianceOperatorEstimation2025}}.
In particular, error analysis is essential for ensuring filter accuracy (i.e., the squared state estimation error is bounded proportionally to the observation noise variance).
\rev{We note that uniform-in-time error bounds are mathematically significant since they guarantee long-time filter accuracy in real world applications, while local-in-time error bounds can lead to blow up of the error in an infinite time horizon.}
For fully observed systems, theoretical bounds for 3DVar and the EnKF have been well studied~\rev{\cite{lawDataAssimilationMathematical2015,kellyWellposednessAccuracyEnsemble2014,takedaUniformErrorBounds2024}}.
In these analyses, a full-rank background covariance is crucial for ensuring filter accuracy, as it guarantees that the filter accounts for uncertainty in all possible directions.
For 3DVar, a trivial choice of the fixed background covariance (e.g., $\Pf = \alpha^2 I$) naturally makes it full-rank.
Because the dynamically updated covariance $\Pf_n$ in the EnKF is potentially rank-deficient, it must be regularized through techniques known as covariance inflation.
For instance, applying additive covariance inflation ($\Pf_n \rightarrow \Pf_n + \alpha^2 I$ for $\alpha > 0$) makes the inflated covariance matrix full-rank, which plays a pivotal role in establishing error bounds for the EnKF~\cite{kellyWellposednessAccuracyEnsemble2014}.

In the partial observation setting, we consider observations obtained via a projection matrix.
Under this setting, the filter accuracy of 3DVar has been established for dissipative dynamical systems~\cite{blomkerAccuracyStabilityContinuoustime2013,lawANALYSIS3DVARFILTER2014}.
In particular, we focus on results for the Lorenz 96 model~\cite{lorenzPredictabilityProblemPartly1996}, which is a chaotic phenomenological model widely used in atmospheric data assimilation.
Law et al. established an error bound of 3DVar for the Lorenz 96 model with a particular \rev{$2/3$-pattern} observation~\cite{lawFilterAccuracyLorenz2016}.
However, error analysis for the EnKF applied to partially observed systems remains limited \rev{\cite{biswasUnifiedFrameworkAnalysis2024,sanz-alonsoLongTimeAccuracyEnsemble2025,calvelloAccuracyEnsembleKalman2026}}.
One of the key challenges is the non-symmetric matrix product $\Pf_n \Proj$ that naturally emerges in the error analysis of Kalman-type filters, where $\Proj$ is the orthogonal projection onto the observation space.
This non-symmetry makes it difficult to estimate the operator norm of matrices involving this product.
\rev{A recent study~\cite{biswasUnifiedFrameworkAnalysis2024} introduces covariance projection as $\Pf_n \rightarrow \Proj \Pf_n \Proj$, which forces the matrix product to be symmetric.
As a result, they established error bounds for the continuous-time versions of the EnKF under the finite-dimensional observations of the two-dimensional Navier-Stokes equations.
Indeed, the covariance projection in the EnKF means neglecting the cross-correlations between the observed and unobserved components.
Moreover, retaining the state-dependent covariance is likely to improve the accuracy guaranteed by their theorems as noted in Remark~5.17 in \cite{biswasUnifiedFrameworkAnalysis2024}.}
\rev{Another study~\cite{sanz-alonsoLongTimeAccuracyEnsemble2025} established an error bound for a deterministic variant of the EnKF applied to abstract dissipative dynamical systems, including the Lorenz 96 model. They achieved this by adding the stochastic perturbations to the prediction ensemble as an alternative to inflation. The covariance of the perturbations is taken to be supported on the observation space, so that it is controlled on the observed subspace.}
\rev{Regardless of whether inflation schemes and covariance projection are used, almost all existing error analyses for partially observed systems rely on strong observation conditions chosen for the target dynamical systems so that errors in unobserved components decay along the time evolution.}

\rev{The aim of this study is to understand how the cross-correlations in the observed--unobserved states affect the EnKF update and performance for partially observed dynamical systems from the viewpoints of the additive inflation and the covariance projection.}
\rev{As a prototypical case}, we focus on a stochastic variant of the EnKF known as the perturbed observation (PO) method \rev{applied to the Lorenz 96 model and rely on the particular $2/3$-pattern observation considered in~\cite{lawFilterAccuracyLorenz2016,sanz-alonsoLongTimeAccuracyEnsemble2025} and the additive covariance inflation}.
As previously discussed, the core challenge lies in handling the non-symmetric matrix product $\Pf_n \Proj$.
Although the covariance projection avoids this non-symmetry entirely, we successfully establish error bounds both with and without the projection.
Specifically, our result offers an extended analytical approach to effectively handle non-symmetric matrices directly without relying on projection of the inflated covariance\rev{, whereas every existing bound relies on confining the covariance or the inflation to the observed subspace \cite{biswasUnifiedFrameworkAnalysis2024,sanz-alonsoLongTimeAccuracyEnsemble2025}.}
\rev{We report two numerical experiments: (EXP1) with the $2/3$-pattern observation for supporting theoretical findings where the unobserved error decays by the dynamics and (EXP2) with a comparatively sparse $1/2$-pattern observation for comparing the projected and non-projected methods in a more practical situation.}

The major contributions of this paper are summarized as follows:
\begin{enumerate}
    \item We establish uniform-in-time error bounds for the PO method\rev{, a stochastic variant of the EnKF, applied to the Lorenz 96 model under the particular $2/3$-pattern observation, both with and without the covariance projection. The projected case (Theorem~\ref{thm:error_bound}) is the discrete-time Lorenz 96 counterpart of \cite{biswasUnifiedFrameworkAnalysis2024}; the non-projected case (Theorem~\ref{thm:error_bound2}) covers the full-rank inflation $\Pf_n + \alpha^2 I$, for which no error bound was previously available.}
    \item We provide operator norm estimates for key non-symmetric matrices (e.g., $\Pf_n \Proj$). This advances the mathematical analysis of Kalman-type filters under partial observations without relying on \rev{the restriction of covariance and inflation to the observed subspace, which is a principal contribution of this study}.
    \item \rev{We report numerical experiments under a comparatively sparse $1/2$-pattern observation, which lies outside the scope of our theorems. There the non-projected method attains higher accuracy, providing quantitative evidence that retaining the observed--unobserved cross-covariance matters when the unobserved error does not decay by the dynamics alone.}
\end{enumerate}

The rest of this paper is organized as follows. Section~\ref{sec:setup} introduces the Lorenz 96 model and the PO method. Section~\ref{sec:results} presents the main theoretical results: the error bounds of the PO method with and without the covariance projection. Section~\ref{sec:numerical} contains numerical experiments validating the theoretical findings. Section~\ref{sec:summary} summarizes the results\rev{, states the limitations of the analysis, and describes open problems}.

\section{Nonlinear data assimilation problem and algorithm}
\label{sec:setup}
\subsection{Notations}
For $ L \in \N $, $ \norm{\bx} $ denotes the standard norm and $ \bx \cdot \bz $ denotes the standard inner product for vectors $ \bx, \bz $ in the Euclidean space $ \R^L $.
We denote the identity matrix on $ \R^L $ by $ I_L \in \R^{L \times L} $.
For $ M \in \N $ and $ A \in \R^{L \times M} $, $ \opnorm{A} $ represents the operator norm of $ A: (\R^{\rev{M}}, \norm{\wc}) \rightarrow (\R^{\rev{L}}, \norm{\wc})$, defined by $ \opnorm{A} \coloneq \sup_{\bx \in \R^{\rev{M}}, \bx \neq \bzero} \frac{\norm{A \bx}}{\norm{\bx}}$.
We also use the notation $ \opnormsub{A}{\R^{\rev{M}}} $ to emphasize the domain.
Similarly, $ \fnorm{A} $ denotes the Frobenius norm of $ A \in \R^{L \times M} $, defined by $ \fnorm{A} \coloneq \sqrt{\sum_{i=1}^L \sum_{j=1}^M A_{i,j}^2} $, where $ A_{i,j} $ is the $ (i,j) $ element of $ A $.
Furthermore, $ \ker(A) $ denotes the kernel of $ A $, $ \ran(A) $ denotes the range of $ A $, and $ A^\top $ denotes the transpose of $ A $.
For $ \bx \in \R^L$ and $ \bz \in \R^M $, we define their product $ \bx \otimes \bz \in \R^{L \times M} $ by $ \bx \otimes \bz: \R^M \ni \bw \mapsto \bx (\bz \cdot \bw) \in \R^L $.
We also define $ \bx \bz^\top \coloneq \bx \otimes \bz$.
Let $ \Sb^L $ denote the space of all symmetric matrices in $ \R^{L \times L} $. 
For $ A \in \Sb^L $, we call $ A $ positive semidefinite if $ \bracket{\bx}{A\bx} \ge 0 $ for all $ \bx \in \R^L $, and we represent it by $ A \succeq 0 $.
Furthermore, if the inequality is strict for all $ \bx \neq \bzero$, we say that $ A $ is positive definite, denoted by $ A \succ 0 $.
For $ A, B \in \Sb^L $, the order $ A \succ (\text{resp.} \succeq) \, B $ means $ A - B \succ (\text{resp.} \succeq) \, 0 $.

For $ N \in \N $, we treat two ensembles of vectors $ \Xe = [\bx^{(k)}]_{k=1}^N $ and $ \Ze = [\bz^{(k)}]_{k=1}^N \in (\R^L)^N$ as matrices in $ \R^{L \times N} $ so that the product $ \Xe \Ze^\top \in \R^{L \times L} $ is given by
\begin{align*}
    \Xe \Ze^\top = \sum_{k=1}^N \bx^{(k)} \otimes \bz^{(k)},
\end{align*}
which is consistent with the matrix product.
Moreover, for $ \bx_0 \in \R^L $, we define a shift
\begin{align*}
    \bx_0 + \Xe &= \bx_0 \bone^\top + \Xe = [\bx_0 + \bx^{(k)}]_{k=1}^N \in \R^{L \times N}.
\end{align*}
For an ensemble $ \Ze = [\bz^{(k)}]_{k=1}^N \in \R^{L \times N} $, $ \overline{\bz} = \frac{1}{N} \sum_{k=1}^N \bz^{(k)} \in \R^L $ is the ensemble mean and $ d\Ze = [\bz^{(k)} - \overline{\bz}]_{k=1}^N \in \R^{L \times N} $ is the ensemble perturbation.
We consider the decomposition of the ensemble $ \Ze $ into its mean and perturbation as $ \Ze = \overline{\bz} + d\Ze $.
The (unbiased) ensemble covariance $ \cov(\Ze) \in \Sb^L $ is defined by
\begin{align*}
    \cov(\Ze) = \frac{1}{N-1} d\Ze d\Ze^\top.
\end{align*}
We denote expectation with respect to random noises in the following observations and algorithms by $ \E $.

\subsection{State space model}
For an integer $ {J} \ge 4 $, we consider the Lorenz 96 model with $ {J} $ components~\cite{lorenzPredictabilityProblemPartly1996}, which is a chaotic phenomenological model widely used in atmospheric data assimilation~\cite{lorenzPredictabilityProblemPartly1996,lorenzOptimalSitesSupplementary1998}.
The governing equation of a state vector $ \bx = (\x^1, \dots, \x^{J})^\top \in \R^{J} $ is given by
\begin{align}
    \label{eq:l96}
    \frac {d\x^j}{dt} = (\x^{j+1} - \x^{j-2}) \x^{j-1} - \x^j + F, \quad j = 1, \dots, {J}, 
\end{align}
where $ \x^{-1} = \x^{{J}-1}$, $ \x^0 = \x^{J} $, and $ \x^{{J}+1} = \x^1$, and $ F \in \R $ is the external forcing.
Since~\eqref{eq:l96} has a unique solution $ \bx(t; \bx_0) $ for any initial state $ \bx_0 \in \R^{J} $, we can define a semigroup $ \Psi_t: \R^{J} \rightarrow \R^{J} $ for $ t \ge 0 $ as $ \Psi_t(\bx_0) = \bx(t; \bx_0) $.
For a fixed interval $ h > 0 $, we denote $ \Psi = \Psi_h $ and consider a discrete dynamical system 
\begin{align}
    \label{eq:dynamics_disc}
    \bx_n = \Psi(\bx_{n-1}), \quad n \in \N .
\end{align}

We consider the following partial observation matrix introduced in~\cite{lawFilterAccuracyLorenz2016}.
\begin{definition}
    Suppose that $ {J} = 3 J' $ for $ J' \in \N $.
    We consider the observation matrix
    \begin{align}
        \label{eq:l96_observation}
        \Hobs = \left[
            \bphi_1,
            \bphi_2, 
            \bphi_4, 
            \bphi_5, 
            \dots,
            \bphi_{3J'-2},  
            \bphi_{3J'-1}
        \right]^\top
        \in \R^{2J' \times {J}},
    \end{align}
    where $ (\bphi_j)_{j=1}^{{J}} $ is the standard basis of $ \R^{{J}} $.
    We also define the associated projection matrix
    \begin{align*}
        \Proj = \Hobs^\top \Hobs = [\bphi_1, \bphi_2, 0, \bphi_4, \bphi_5, 0, \dots] \in \R^{{J} \times {J}}.
    \end{align*}
\end{definition}
Using $ \Proj $, we define the observed-unobserved decomposition of the state vector as
\begin{align*}
    \R^{J} = \ran{\Proj} \oplus \ran(I-\Proj) = \ran{\Proj} \oplus \ker{\Proj}.
\end{align*}
Noisy and partial observations $ (\by_n)_{n \in \N} $ are obtained in the observation space $ \R^{2J'} \simeq \Hobs(\R^{J}) $ as 
\begin{align}
    \label{eq:obs}
    \by_n = \Hobs \bx_n + \bxi_n, \quad \bxi_n \overset{\mathrm{i.i.d.}}{\sim} \mathcal{N}(0, \Robs),
\end{align}
where $ \Robs \in \Sb^{2J'} $ is a covariance matrix with $ \Robs \succ 0 $, and $ \mathcal{N}(0, \Robs) $ denotes the Gaussian distribution with mean $ 0 $ and covariance $ \Robs $.
The random noise $ \bxi_n $ represents the measurement error.
\begin{remark}
    \label{rem:l96_observation}
    The choice of $ \Hobs $ in \eqref{eq:l96_observation} is motivated by the fact that the components in $ \ran{\Proj} $ are dominant in the dynamics of the Lorenz 96 model and the components in $ \ker{\Proj} $ are slaved to the components in $ \ran{\Proj} $, which is expressed mathematically in \cite[Theorem 4.5]{lawFilterAccuracyLorenz2016}.
\end{remark}

In the current study, we consider the filtering problem for the state space model \eqref{eq:dynamics_disc} and \eqref{eq:obs}.
\rev{That is, we estimate} the true state $ \bx_n $ from the noisy and partial observation sequence $ Y_n = (\by_k)_{k=1}^n $.
We employ the Bayesian approach to the filtering problem, i.e., we aim to compute the filtering distribution as the posterior distribution $ \Prob^{\bx_n}( \cdot \mid Y_n) $ for $ n \in \N $.
The actual algorithm to approximate this filtering distribution is described later.
To utilize the analysis of dissipative dynamical systems, we rewrite~\eqref{eq:l96} in the following form 
\begin{align}
    \label{eq:l96_dissipative}
    \frac{d\bx}{dt} + \A \bx + \B(\bx, \bx) = \bm{f},
\end{align}
where
\begin{align*}
  \A = I_{J}, \quad
  \B(\bx, \bz) = \frac{1}{2}\left[
      \begin{array}{c}
        \z^2 \x^{J} + \x^2 \z^{J} - \z^{J} \x^{J-1} - \x^{J} \z^{J-1}\\
        \vdots\\
        \z^{i-1} \x^{i+1} + \x^{i-1} \z^{i+1} - \z^{i-2} \x^{i-1} - \x^{i-2} \z^{i-1}\\
        \vdots\\
        \z^{J-1} \x^1 + \x^{J-1} \z^1 - \z^{J-2} \x^{J-1} - \x^{J-2} \z^{J-1}
      \end{array}
    \right] \in \R^{J},
    \quad \bm{f} = (F, \dots, F)^\top \in \R^{J}
\end{align*}
for $ \bx, \bz \in \R^{J} $.
The following lemmas characterize the dissipative and chaotic properties of the Lorenz 96 model \cite{lawFilterAccuracyLorenz2016}.
\begin{lemma}[\cite{lawFilterAccuracyLorenz2016}]
    \label{lem:l96_absorbing}
    Let $ \rho = \sqrt{2J}|F|$ and $ B(\rho) \coloneq \{\bx \in \R^J \mid \norm{\bx} \le \rho\} $.
    Then, for any $ \bx_0 \in \R^J $, there exists $ T = T(|\bx_0|) \ge 0 $ such that
    $ \bx(t) \in B(\rho) $ for the solution $\bx(t)$ to~\eqref{eq:l96_dissipative} with $ \bx_0 $ and for all $ t \ge T$.
\end{lemma}
\begin{lemma}[\cite{lawFilterAccuracyLorenz2016}]
    \label{lem:l96_chaos}
    Let $ \bx $ and $ \bz $ be two solutions to~\eqref{eq:l96_dissipative} with $ \bx_0 \in B(\rho)$ and $ \bz_0 \in \R^J $, respectively.
    Then, there exists $ \beta \in \R$ such that
    \begin{align}
        \label{eq:l96_chaos}
        \norm{\bx - \bz}^2 \le \norm{\bx_0 - \bz_0}^2 \exp(2\beta t).
    \end{align}
\end{lemma}

\begin{remark}
    Lemma~\ref{lem:l96_absorbing} describes the dissipative property of the Lorenz 96 model.
    On the other hand, \eqref{eq:l96_chaos} in Lemma~\ref{lem:l96_chaos} provides a bound on the exponential separation rate of two trajectories due to the chaotic nature of the Lorenz 96 model when $\beta > 0$.
    The bound on the growth rate $\beta$ depends on $J$ and $F$.
\end{remark}

\subsection{The ensemble Kalman filter}
Let $ \Ne \in \N $ be an ensemble size.
To approximate the filtering distribution, the ensemble Kalman filter (EnKF) uses the empirical distribution of an ensemble of particles $ \Ze_n = [\bz^{(1)}_n, \dots, \bz^{(\rev{\Ne})}_n] \in \R^{J \times \Ne} $, $ \Prob^{\bx_n}(\wc|Y_n) \approx \frac{1}{\Ne} \sum_{k=1}^{\Ne} \delta_{\bz^{(k)}_n}(\wc) $, where $ \delta_a(\wc) $ denotes the Dirac measure.
For each time step $n \in \N$, we perform two steps: the prediction step and the analysis step.
In the prediction step, each particle is driven by the dynamics~\eqref{eq:dynamics_disc}.
In the analysis step, the ensemble is corrected using the observation data based on the least squares \rev{method}.
A stochastic version of the EnKF is the perturbed observation (PO) method~\cite{burgersAnalysisSchemeEnsemble1998}.
A detailed description is given below.
\begin{definition}
  \label{def:po}
 Let $ \Ze_0 = [\bz^{(k)}_0]_{k=1}^{\Ne} \in \R^{J \times \Ne} $.
 The perturbed observation (PO) method recursively \rev{generates} the analysis ensemble $\Ze_n$ by the following two steps.
 For $ n \in \N $,
  \begin{enumerate}[label=(\Roman*)]
  \item {\label{def:po:1} (Prediction: $ \Ze_{n-1}  \rightarrow \Zehat_n $)
  Evolve each member by the dynamics \eqref{eq:dynamics_disc},
      \begin{align*}
         \bzf_n^{(k)} = \Mdyn(\bz_{n-1}^{(k)}),
      \end{align*}
 for $k = 1, \dots, \Ne$ and set $ \Zehat_n = [\bzf_n^{(k)}]_{k=1}^{\Ne} \in \R^{J \times \Ne} $.
 }
   \item {\label{def:po:2} (Analysis: $ \Zehat_n, \by_n \rightarrow \Ze_n $) 
   Set $ \Pf_n = \cov(\Zehat_n) $ and replicate observations by adding random perturbations (artificial noises),
    \begin{align}
        \label{eq:po:perturbation}
        \by_n^{(k)} = \by_n + \bxi_n^{(k)}, \quad \bxi_n^{(k)} \overset{\mathrm{i.i.d.}}{\sim} \mathcal{N}(0, \Robs),
   \end{align}
   for $k = 1, \dots, \Ne$  and update the ensemble,
   \begin{align}
      \label{eq:po:update}
      \bz_n^{(k)} = \bzf_n^{(k)} + K_n(\by_n^{(k)} - \Hobs \zf_n^{(k)}),
   \end{align}
   for $k = 1, \dots, \Ne$ with the Kalman gain 
   \begin{align}
      \label{eq:kalman_gain_po}
      K_n = \Pf_n\Hobs^\top (\Hobs\Pf_n\Hobs^\top + \Robs)^{-1}.
   \end{align}
 Finally, set $ \Ze_n = [\bz_n^{(k)}]_{k=1}^{\Ne} \in \R^{J \times \Ne} $.
 For given $ \by_n, \Pf_n $, we denote the (random) map from $ \Zehat_n $ to $ \Ze_n $ as $ \Ze_n = \Ze_{PO}(\Zehat_n; \by_n, \Pf_n) $ even though it depends on the artificial noise sequence $ (\bxi_n^{(k)})_{n \in \N, k \in \{1, \dots, \Ne\}} $.
 }
  \end{enumerate}
\end{definition}
We generate the initial ensemble $ \Ze_0 $ randomly in the following numerical experiments and assume that it is independent of the observation noise sequence $ (\bxi_n)_{n \in \N} $.
We also assume that the artificial noise sequence $ (\bxi_n^{(k)})_{n \in \N, k \in \{1, \dots, \Ne\}} $ in \eqref{eq:po:perturbation} is independent of the initial ensemble and the observation noise sequence.



We introduce numerical techniques to stabilize the PO method, known as covariance inflation methods.
\begin{definition}
    \label{def:inflation_proj}
    Let $ \alpha > 0$ be an inflation parameter.
    In the analysis step~\ref{def:po:2} of~Definition~\ref{def:po}, we first compute the inflated covariance
    \begin{align}
        \label{eq:inflated_covariance}
        \Pf_n^{\alpha} = \Pf_n + \alpha^2 I_{J} \in \R^{J \times J}
    \end{align}
    or the inflated and projected covariance
    \begin{align}
        \label{eq:inflated_projected_covariance}
        \Pf_n^{\alpha} = \Proj(\Pf_n + \alpha^2 I_{J}) \Proj \in \R^{J \times J}.
    \end{align}
    Then, we compute the analysis ensemble as
    \begin{align*}
        \Ze_n = \Ze_{PO}(\Zehat_n; \by_n, \Pf_n^{\alpha}).
    \end{align*}
\end{definition}
Using the projected covariance inflation~\eqref{eq:inflated_projected_covariance} implies that we ignore correlations between observed and unobserved subspaces, and we correct the error only in the observed subspace in the analysis step. 
\begin{revision}
    The inflated covariances used in existing accuracy results for the EnKF with partial observations are compared in Table~\ref{tbl:covariance_designs}.
    \begin{table}[pos=!ht]
    \caption{Inflated covariance used in accuracy results for the EnKF with partial observations.
    Here $\Proj$ is the orthogonal projection onto the observation space, and the forecast covariance matrices are written in the notation of the current study up to a constant factor.}
    \label{tbl:covariance_designs}
    \begin{tabular*}{\tblwidth}{@{}LLLL@{}}
        \toprule
        Method & Inflated covariance & Inflation & Restricted to the observed subspace \\
        \midrule
        (5.51) in \cite{biswasUnifiedFrameworkAnalysis2024} & $\Proj \Pf_n \Proj + \mu \Proj, \mu > 0$ & Deterministic, additive & Covariance and inflation \\
        (2.6) in \cite{sanz-alonsoLongTimeAccuracyEnsemble2025} & $\Pf_n + a \Proj, a > 0$ & Stochastic, by perturbing each prediction member & Inflation \\
        add-proj \eqref{eq:inflated_projected_covariance} & $\Proj(\Pf_n + \alpha^2 I_J)\Proj$ & Deterministic, additive & Covariance and inflation \\
        add \eqref{eq:inflated_covariance} & $\Pf_n + \alpha^2 I_J$ & Deterministic, additive & Neither \\
        \bottomrule
    \end{tabular*}
\end{table}
\end{revision}

\section{Error analysis}
\label{sec:results}

Let $ \berr_n^{(k)} = \bz_n^{(k)} - \bx_n $.
We define another norm $ \pnorm{\bz} = \sqrt{\norm{\bz}^2 + \norm{\Proj \bz}^2} $ associated with $ \Proj $ for $ \bz \in \R^J $.

\subsection{Error bound with the covariance projection}
We derive an error bound for the PO method with the projected covariance inflation.
\begin{theorem}[Error bound of the projected PO method]
    \label{thm:error_bound}
    Let us consider the Lorenz 96 model~\eqref{eq:l96_dissipative}, the observation matrix~\eqref{eq:l96_observation}, and the PO method with the projected covariance inflation~\eqref{eq:inflated_projected_covariance}.
    Suppose that $ \Robs = r^2 I_{2J'} $ for $ r > 0 $, and $ \bz^{(k)}_n \in B(\rho) $ for $ k = 1, \dots, \Ne $ and $ n \in \N $.
    Then, there exist $ \tau > 0 $, $ \alpha > 0 $, and $ \theta \in (0, 1) $ such that
    \begin{align}
        \label{eq:error_bound}
        \E\left[\pnorm{\berr^{(k)}_n}^2\right] \le \theta^n \E\left[\pnorm{\berr^{(k)}_0}^2\right] + 4 \rev{\Ny} r^2 \frac{1 - \theta^n}{1 - \theta}
    \end{align}
    for $ k = 1, \dots, \Ne $ and $ n \in \N $, where $ \Ny = 2J'$.
    Furthermore, it follows that
    \begin{align}
        \label{eq:error_bound_limit}
        \limsup_{n\rightarrow\infty}\E\left[\pnorm{\berr^{(k)}_n}^2\right] \le \frac{4 \rev{\Ny} r^2}{1 - \theta}.
    \end{align}
\end{theorem}

To prove~Theorem~\ref{thm:error_bound}, we estimate the error growth in the prediction step using the following lemma.
The key is to estimate the error growth over the observed space and the whole space separately.
\begin{lemma}
    \label{lem:pred_bound}
    Under the conditions in~Theorem~\ref{thm:error_bound}, let $ \berr(t) = \Mdyn_t(\bz_0) - \Mdyn_t(\bx_0) $ for any $ \bx_0, \bz_0 \in B(\rho) $ and $ t \ge 0 $.
    Then, we have
    \begin{align}
        \norm{\Proj \berr(t)}^2 & \le a_1(t)\norm{\berr(0)}^2 + \norm{\Proj \berr(0)}^2 \label{eq:pred_bound_proj},\\
        \norm{\berr(t)}^2 & \le b_1(t)\norm{\berr(0)}^2 + b_2(t)\norm{\Proj \berr(0)}^2 \label{eq:pred_bound}
    \end{align}
    for $ t \ge 0 $, where
    \begin{align}
        a_1(t) & = \frac{16\rho^2}{\beta} (\exp(2\beta t) - 1), \label{eq:a1} \\
        b_1(t) & = \frac{16c^2\rho^4}{\beta} \left[\frac{1}{2\beta + 1}(\exp(2\beta t) - \exp(-t)) - (1 - \exp(-t))\right] + \exp(-t), \label{eq:b1} \\
        b_2(t) & = c^2 \rho^2 (1 - \exp(-t)), \label{eq:b2} 
    \end{align}
    and $ c > 0 $.
    Furthermore, taking $ t > 0 $ small enough, we obtain
    \begin{align}
        \label{eq:pred_contraction}
        a_1(t), b_1(t), b_2(t) < 1.
    \end{align}
\end{lemma}
The proof of~Lemma~\ref{lem:pred_bound} is moved to Appendix~\ref{sec:appendix_proofs}.
The coefficients $a_1(t), b_1(t), b_2(t)$ represent \rev{bounds on the error propagation between} the observed space and the whole space.

We estimate the error contraction in the analysis step by the following lemma.
\begin{lemma}
    \label{lem:analysis_bound}
    Under the conditions in~Theorem~\ref{thm:error_bound}, let $ \Theta = \left(\frac{\robs^2}{\robs^2 + \alpha^2}\right)^2 $.
    Then, we have
    \begin{align}
        \E\left[\norm{\Proj \berr^{(k)}_n}^2\right] & \le \Theta(\alpha) \E\left[\norm{\Proj \berrhat^{(k)}_n}^2\right] + 2 \rev{\Ny} r^2 \label{eq:analysis_bound_proj}, \\
        \E\left[\norm{\berr^{(k)}_n}^2\right] & \le \E\left[\norm{\berrhat^{(k)}_n}^2\right] + 2 \rev{\Ny} r^2 \label{eq:analysis_bound}
    \end{align}
    for $ k = 1, \dots, \Ne $, and $ n \in \N $ where $ \Ny = 2J'$.
\end{lemma}
The proof of~Lemma~\ref{lem:analysis_bound} is moved to Appendix~\ref{sec:appendix_proofs}.
The coefficient $\Theta$ represents the error contraction rate in the observed space in the analysis step.
\rev{In the whole space, on the other hand, only the absence of error expansion is guaranteed.}

\begin{proof}[Proof of~Theorem~\ref{thm:error_bound}]
    From~Lemma~\ref{lem:analysis_bound}, we have
    \begin{align}
        \label{eq:proof1}
        \E\left[\pnorm{\berr_n^{(k)}}^2\right] \le \E\left[\norm{\berrhat^{(k)}_n}^2\right] + \Theta \E\left[\norm{\Proj \berrhat^{(k)}_n}^2\right] + 4 \rev{\Ny} r^2.
    \end{align}
    Since $ \bz_n^{(k)} \in B(\rho) $ for $ k = 1, \dots, \Ne$ \rev{and $n \in \N$} by assumption, taking $ \bz_0 = \bz_{n-1}^{(k)}$ in~Lemma~\ref{lem:pred_bound} and considering~\eqref{eq:proof1} yield
    \begin{align*}
        & \E\left[\pnorm{\berr_n^{(k)}}^2\right] \le b_1(\tau)\E\left[\norm{\berr_{n-1}^{(k)}}^2\right] + b_2(\tau)\E\left[\norm{\Proj \berr_{n-1}^{(k)}}^2\right] + \Theta(a_1(\tau)\E\left[\norm{\berr_{n-1}^{(k)}}^2\right] + \E\left[\norm{\Proj \berr_{n-1}^{(k)}}^2\right]) + 4\rev{\Ny}r^2 \\
        & = (\Theta a_1(\tau) + b_1(\tau))\E\left[\norm{\berr_{n-1}^{(k)}}^2\right] + (\Theta + b_2(\tau))\E\left[\norm{\Proj \berr_{n-1}^{(k)}}^2\right] \\
        & \le \theta \E\left[\pnorm{\berr_{\rev{n-1}}^{(k)}}^2\right] + 4\rev{\Ny}r^2,
    \end{align*}
    where $ \theta \coloneqq \max\{\Theta(\alpha) a_1(\tau) + b_1(\tau),\Theta(\alpha) + b_2(\tau)\} $.
    From~\eqref{eq:pred_contraction}, taking $ \tau > 0 $ small enough and $ \alpha > 0 $ large enough so that $ \theta < 1 $,
    we have
    \begin{align*}
        \E\left[\pnorm{\berr^{(k)}_n}^2\right] \le \theta \E\left[\pnorm{\berr^{(k)}_{n-1}}^2\right] + 4 \rev{\Ny} r^2.
    \end{align*}
    This gives~\eqref{eq:error_bound}.
    Furthermore, taking the limit $n \rightarrow \infty$ yields~\eqref{eq:error_bound_limit}.
\end{proof}

\begin{corollary}
\label{cor:order_r}
Under the conditions of~Theorem~\ref{thm:error_bound}, we consider the accurate observation limit $ \robs \rightarrow 0$.
Then, we have 
\begin{align}
    \label{eq:order_r}
    \limsup_{n \rightarrow \infty} \E\left[\pnorm{\berr_n^{(k)}}^2\right] = \mathcal{O}(\robs^2).
\end{align}
\end{corollary}
\begin{proof}
    \begin{revision}
    We can choose $\tau > 0 $ so that $b \coloneq \max\{b_1(\tau), b_2(\tau)\} < 1$ and $a_1(\tau)<1$.
    We fix any $r_0 > 0$.
    Then, we define $c = 1 - b > 0$ and choose $\alpha = \alpha(r_0) > 0$ so that $\Theta(r_0, \alpha) < c$, where $\Theta(r, \alpha) \coloneq \left(\frac{r^2}{r^2 + \alpha^2}\right)^2$.
    For fixed $\alpha>0$, $\Theta(r,\alpha)$ is an increasing function of $r>0$.
    Hence, it follows that
    \begin{align*}
        \theta(r) = \max\{\Theta(r, \alpha) a_1(\tau) + b_1(\tau),\Theta(r, \alpha) + b_2(\tau)\} \le b + \Theta(r_0, \alpha) = 1 - \epsilon,
    \end{align*}
    where $ \epsilon = c - \Theta(r_0, \alpha) > 0$.
    Therefore, we have $1 - \theta(r) \ge \epsilon > 0$ when $0 < r \le r_0$.
    \end{revision}
    Then, the \rev{right-hand side} of~\eqref{eq:error_bound_limit} is reduced to
    \begin{revision}
    \begin{align*}
        \frac{4 \Ny r^2}{1 - \theta(r)} \le \frac{4 \Ny r^2}{\epsilon} = \mathcal{O}(r^2),
    \end{align*}
    \end{revision}
    which yields \eqref{eq:order_r}.
\end{proof}
\begin{revision}
    \rev{If $\theta$ is sufficiently small so that $\epsilon \approx 1$, then Corollary~\ref{cor:order_r} allows us to control} the expectation of the long-time state estimation error within $\delta_{\mathrm{target}} > 0$ by taking the observation noise variance $r^2$ sufficiently small such that
\begin{align}
    \label{eq:error_control}
    r^2 < \frac{\delta_{\mathrm{target}}}{4\rev{\Ny}}.
\end{align}

\end{revision}
The symmetry of $ \Pf_n^\alpha$ plays a major role in the proof of~Lemma~\ref{lem:analysis_bound}, and consequently in the proof of~Theorem~\ref{thm:error_bound}.
From~\eqref{eq:kalman_gain_po}, we need to estimate the operator norm $ \opnorm{(I+A)^{-1}}$ for $ A = \Pf_n \Hobs^\top \Robs^{-1} \Hobs$ in the error analysis of the EnKF.
In the proof of~Theorem~\ref{thm:error_bound}, $ A $ is reduced to $ r^{-2}\Pf_n^\alpha = r^{-2} \Proj (\Pf_n + \alpha^2 I_{J}) \Proj $ owing to the projection in~Definition~\ref{def:inflation_proj}.
However, this matrix $A$ is not symmetric in general.
In such a case, the operator norm $\opnorm{(I+A)^{-1}}$ \rev{can exceed $1$, i.e., the update is not necessarily contracting}.
The following remark provides a typical example that naturally emerges from the matrix product involving a projection matrix.
\begin{remark}[Non-symmetric operator norm]
    \label{rem:nonsymmetric}
    Let us consider a matrix $A \in \R^{2 \times 2}$ as
    \begin{align*}
        A = \left(
        \begin{array}{cc}
            1 & * \\
            z & *
        \end{array}
        \right)
        \left(
        \begin{array}{cc}
            1 & 0 \\
            0 & 0
        \end{array}
        \right)
        = \left(
        \begin{array}{cc}
            1 & 0 \\
            z & 0
        \end{array}
        \right),
    \end{align*}
    where $z \neq 0$.
    Then, it follows that 
    \begin{align*}
        \opnorm{(I + A)^{-1}} > 1.
    \end{align*}
    Indeed, this can be verified as follows.
    From the formula of an invertible $2\times 2$ matrix, we have
    \begin{align*}
        (I + A)^{-1}
        = \frac{1}{2} \left(
            \begin{array}{cc}
                1 & 0 \\
                -z & 2
            \end{array}
            \right).
    \end{align*}
    In general, the operator norm of a matrix $B$ is given by the largest singular value $s_1(B)$, which is the square root of the largest eigenvalue of $B^\top B$.
    For $B = (I+A)^{-1}$, we have
    \begin{align*}
        [(I + A)^{-1}]^\top(I + A)^{-1} = 
        \frac{1}{4} \left(
            \begin{array}{cc}
                1 & -z \\
                -z & z^2+4
            \end{array}
            \right).
    \end{align*}
    An eigenvalue $\lambda$ of this matrix satisfies the characteristic equation
    \begin{align*}
        4\lambda^2 - (z^2+5)\lambda + 1 = 0.
    \end{align*}
    The larger solution $\lambda_+$ is given by
    \begin{align*}
        \lambda_+ = \frac{z^2+ 5 + \sqrt{(z^2+5)^2 - 16}}{8} > 1
    \end{align*}
    for any $z \neq 0$.
    This implies that $\opnorm{(I+A)^{-1}} > 1$.
\end{remark}
\rev{We note that the non-symmetric structure in Remark~\ref{rem:nonsymmetric} is unavoidable in the Kalman-update \eqref{eq:po:update} with non-trivial covariance matrices and partial observations. On the other hand, it does not depend on the dynamics.}

\subsection{Error bound without the covariance projection}
In this section, we provide an error bound for the PO method without the covariance projection in the additive inflation step, i.e., $\Pf_n^\alpha = \Pf_n + \alpha^2 I_{J}$.
The major obstacle in the analysis is the treatment of the non-symmetric matrix $I_{J} + r^{-2} \Pf_n^\alpha \Pi$.
\begin{theorem}[Error bound of the non-projected PO method]
    \label{thm:error_bound2}
    Let us consider the Lorenz 96 model~\eqref{eq:l96_dissipative}, the observation matrix~\eqref{eq:l96_observation}, and the PO method with the covariance inflation~\eqref{eq:inflated_covariance}.
    Suppose that $ \Robs = r^2 I_{2J'} $ for $ r > 0 $, and $ \bz^{(k)}_n \in B(\rho) $ for $ k = 1, \dots, \Ne $ and $ n \in \N $.
    Then, there exist $ \tau > 0 $, $ \alpha > 0 $, and $ \theta \in (0, 1) $ such that
    \begin{align}
        \label{eq:error_bound2}
        \E\left[\pnorm{\berr^{(k)}_n}^2\right] \le \theta^n \E\left[\pnorm{\berr^{(k)}_0}^2\right] + 4 \rev{\Ny} r^2 \frac{1 - \theta^n}{1 - \theta}
    \end{align}
    for $ k = 1, \dots, \Ne $ and $ n \in \N $, where $ \Ny = 2J'$.
    Furthermore, it follows that
    \begin{align}
        \label{eq:error_bound_limit2}
        \limsup_{n\rightarrow\infty}\E\left[\pnorm{\berr^{(k)}_n}^2\right] \le \frac{4 \rev{\Ny} r^2}{1 - \theta}.
    \end{align}
\end{theorem}
We need to establish an alternative to~Lemma~\ref{lem:analysis_bound} due to the use of a different inflation method.
\begin{lemma}
    \label{lem:analysis_bound2}
    Under the conditions in~Theorem~\ref{thm:error_bound2}, let $ \Theta = \left(\frac{\robs^2}{\robs^2 + \alpha^2}\right)^2 $.
    Then, there is $\widetilde{\Theta} \ge 1$ such that
    \begin{align}
        \E\left[\norm{\Proj \berr^{(k)}_n}^2\right] & \le \Theta(\alpha) \E\left[\norm{\Proj \berrhat^{(k)}_n}^2\right] + 2 \rev{\Ny} r^2 \label{eq:analysis_bound_proj2}, \\
        \E\left[\norm{\berr^{(k)}_n}^2\right] & \le \widetilde{\Theta}(\alpha) \E\left[\norm{\berrhat^{(k)}_n}^2\right] + 2 \rev{\Ny} r^2 \label{eq:analysis_bound2}
    \end{align}
    for $ k = 1, \dots, \Ne $, and $ n \in \N $ where $ \Ny = 2J'$.
    Moreover, $\widetilde{\Theta}$ monotonically decreases to $1$ as $\alpha \rightarrow \infty$.
\end{lemma}
The proof of Lemma~\ref{lem:analysis_bound2} is moved to Appendix~\ref{sec:appendix_proofs}.

\begin{proof}[Proof of~Theorem~\ref{thm:error_bound2}]
    The main modification from the proof of~Theorem~\ref{thm:error_bound} is the definition of $\theta$:
    \begin{align*}
        \theta \coloneqq \max\{\Theta(\alpha) a_1(\tau) + \widetilde{\Theta}(\alpha)b_1(\tau),\Theta(\alpha) + \rev{\widetilde{\Theta}(\alpha)}b_2(\tau)\},
    \end{align*}
    where $ a_1(\tau)$, $ b_1(\tau)$, $b_2(\tau)$, $\Theta(\alpha)$, and $\widetilde{\Theta}(\alpha)$ are from~Lemmas~\ref{lem:pred_bound} and \ref{lem:analysis_bound2}.
    Taking sufficiently small $\tau > 0$ yields that $b_1(\tau),b_2(\tau) < 1$ from~Lemma~\ref{lem:pred_bound}.
    Therefore, we have $\widetilde{\Theta}(\alpha)b_1(\tau) < 1$ \rev{and $\widetilde{\Theta}(\alpha)b_2(\tau) < 1$} for sufficiently large $\alpha$ from~Lemma~\ref{lem:analysis_bound2}.
    Moreover, since $\Theta(\alpha) $ monotonically decreases to $0$ as $\alpha \rightarrow \infty$, it follows that $\Theta(\alpha) a_1(\tau) + \widetilde{\Theta}(\alpha)b_1(\tau) < 1$ and $\Theta(\alpha) + \rev{\widetilde{\Theta}(\alpha)}b_2(\tau) < 1$.
    This implies that $\theta < 1$ for sufficiently small $\tau > 0$ and sufficiently large $\alpha>0$.
\end{proof}

\begin{remark}
    \label{rem:order_r}
    The error bounds \eqref{eq:error_bound2} and \eqref{eq:error_bound_limit2} in Theorem~\ref{thm:error_bound2} are qualitatively same as those in Theorem~\ref{thm:error_bound}.
    Consequently, we have the $\mathcal{O}(r^2)$-bound as \eqref{eq:order_r} for the PO method with the non-projected additive covariance inflation.
\end{remark}
\rev{Based on Remark~\ref{rem:order_r}, Corollary~\ref{cor:order_r}, the error control \eqref{eq:error_control} as well as Theorems~\ref{thm:error_bound} and \ref{thm:error_bound2}, our theory guarantee a reliable state estimation by the projected and non-projected PO method under a particular partial observation \eqref{eq:l96_observation}.}

\section{Numerical example\rev{s}}
\label{sec:numerical}
\begin{revision}
    \subsection{Experimental setup}
\end{revision}
For the model parameters $ J = 60 $ and $ F = 8.0 $, we approximate the time evolution $\Psi_\tau$ of the Lorenz 96 model~\eqref{eq:l96} using the fourth-order Runge-Kutta method with a time step size of $\Delta t = 0.01 $.
We fix the observation interval $ \tau = 0.01 $ and compute the true trajectory $ (\bx_n)_{n=1}^{N_t}$ according to~\eqref{eq:dynamics_disc} with $ \bx_0 \in B(\rho)$ and $ N_t = \rev{1}000 $.
\begin{revision}
To support theoretical findings and explore numerical performance beyond theory, we perform two numerical experiments using the following two observation matrices.
\begin{enumerate}
    \item[(EXP1)] Comparatively dense observation, $\Ny=40$: partial observation matrix~\eqref{eq:l96_observation} used in theoretical analysis above.
    \item[(EXP2)] Comparatively sparse observation, $\Ny=30$: 
    \begin{align}
        \label{eq:l96_observation_sparse}
        \Hobs = \left[
            \bphi_1,
            \bphi_3, 
            \bphi_5, 
            \dots,
            \bphi_{59}
            \right]^\top
            \in \R^{30 \times {60}}.
        \end{align}
    \end{enumerate}
\end{revision}
\rev{For both observation matrices \eqref{eq:l96_observation} and \eqref{eq:l96_observation_sparse}, we consider the trivial observation noise covariance matrix} $ \Robs = \robs^2 I$ \rev{with} $\robs = 1.0$ \rev{and} generate the observation sequence $ (\by_n)_{n=1}^{N_t} $ according to~\eqref{eq:obs}.
We apply the PO method with the additive inflation~\eqref{eq:inflated_covariance} and the projected additive inflation~\eqref{eq:inflated_projected_covariance} \rev{using} $\Ne = 10$ and $ \alpha=0.0, 0.5, 2.0$ \rev{for (EXP1) and $\Ne=10, 20, 40, 80$ and $\alpha=0.5$ for (EXP2)}.
The initial ensemble is generated by adding $ \rev{\Ne} $ Gaussian noises with covariance $ 16 I_{60} $ to a randomly chosen point from the attractor of the Lorenz 96 model.
\begin{revision}
    When evaluating the expectation $\E$ of a random variable, we approximate it by the sample mean over 20 independent realizations with different random seeds. The randomness arises from the initial ensemble, the observation noise, and the perturbations used in the PO method, while the true trajectory is kept fixed.
    The variability between random seeds is displayed by the 95\% interval (the empirical 2.5th--97.5th percentile range, not a confidence interval).
    We summarize all parameters used in the numerical experiments in Table~\ref{tbl:params}.
    \begin{table}[pos=!ht]
    \caption{Numerical parameters used in the experiments}\label{tbl:params}
    \begin{tabular*}{\tblwidth}{@{}LLL@{}}
    \toprule
    Name & Value & Description \\ 
    \midrule
    State dimension, $J$ & $60$ & Number of Lorenz 96 state variables \\
    Forcing, $F$ & $8.0$ & Lorenz 96 forcing parameter \\
    Integration step, $\Delta t$ & $0.01$ & Time step of the numerical integrator \\
    Observation interval, $\tau$ & $0.01$ & Time between consecutive observations \\
    Assimilation steps, $N_t$ & $1000$ & Number of analysis times \\
    Observation dimension, $\Ny$ & $40$, $30$ & Dense and sparse observation patterns, respectively \\
    Observation noise, $\robs$ & $1.0$ & Standard deviation in $\Robs=\robs^2 I_{\Ny}$ \\
    Ensemble size, $\Ne$ & $10$ (\textrm{EXP1}) & Number of ensemble members \\
                         & $10,\,20,\,40,\,80$ (\textrm{EXP2})&\\
    Inflation, $\alpha$ & $0,\,0.5,\,2,\,10,\,100$ (\textrm{EXP1}) & Additive inflation parameter \\
                        & $0.5$ (\textrm{EXP2}) & \\
    Initial covariance & $16I_{60}$ & Covariance of the initial ensemble perturbations \\
    Independent runs & $20$ & Random seeds $0,\ldots,19$ \\
    \bottomrule
    \end{tabular*}
    \end{table}
\end{revision}

\begin{revision}
    We list assumptions imposed in numerical experiments:
    \begin{enumerate}[label=(\roman*), ref=(\roman*)]
        \item The state $\bx(t)$ after sufficiently long time evolution belongs to $B(\rho)$. We discard a spin-up $20\times360$ steps to obtain the initial state $\bx_0$ of the true trajectory.
        \item The ensemble is bounded: $\bz_n^{(k)} \in B(\rho), \, k=1,\dots,\Ne$ while an algorithm is working.
        \item The value $\theta>0$ in theorems is sufficiently small for sufficiently large $\alpha>0$.
        It is needed to evaluate a theoretical error bound since theoretical estimates for related parameters are conservative.
        \item Observation interval $\tau>0$ is sufficiently small and inflation $\alpha>0$ is sufficiently large.
    \end{enumerate}
\end{revision}

\begin{revision}
    \subsection{Comparatively dense observation (EXP1)}
    In this section, we use the comparatively dense observation matrix \eqref{eq:l96_observation} to support theoretical findings.
\end{revision}
Figure~\ref{fig:se} shows the time series of the Mean Squared Error (MSE) $ \frac{1}{m} \sum_{k=1}^m \E\left[\pnorm{\berr^{(k)}_n}^2\right] $.
In practice, it is difficult to evaluate $\theta$.
Instead, we assume that $\theta$ is sufficiently small, provided that the inflation method effectively suppresses the error growth in the prediction step.
Then, we have \rev{an estimated} upper bound $ 4 N_y r^2$ in \eqref{eq:error_bound_limit} by ignoring $ \theta$.
The bound is indicated by a black line in Figure~\ref{fig:se}.
\begin{figure}[pos=!ht]
    \centering
    \includegraphics[width=0.75\columnwidth]{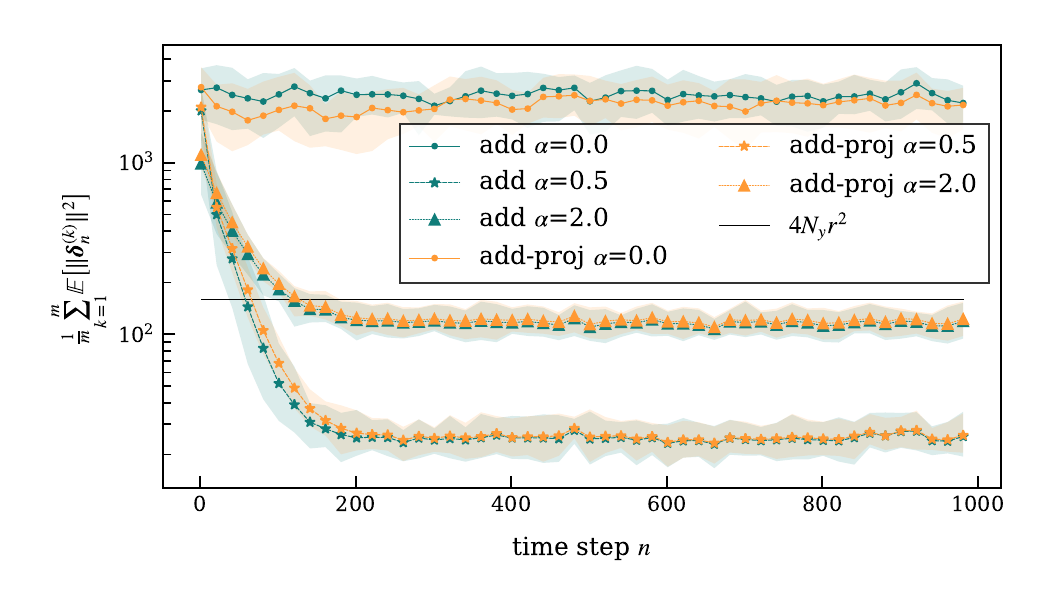}
    \caption{Time series of the Mean Squared Error $ \frac{1}{m} \sum_{k=1}^m \E\left[\pnorm{\berr^{(k)}_n}^2\right] $ with a log-scale y-axis. The covariance inflation methods are the additive inflation~\eqref{eq:inflated_covariance} (add) and the projected additive inflation~\eqref{eq:inflated_projected_covariance} (add-proj) with $ \alpha = 0.0, 0.5, 2.0 $. \rev{The 95\% intervals are shown in shaded color.}}
    \label{fig:se}
\end{figure}
\begin{revision}
    \begin{table}[pos=!ht]
        \caption{Time-averaged MSEs corresponding to Figure~\ref{fig:se}. For each seed, the error is averaged over $n=501,\ldots,1000$. The MSE is the mean over the 20 seed-wise time averages, and $p_{2.5}$ and $p_{97.5}$ are their empirical percentiles.}
        \label{tbl:fig1_mse}
        \begin{tabular*}{\tblwidth}{@{}LCCCC@{}}
            \toprule
            Method & $\alpha$ & MSE & $p_{2.5}$ & $p_{97.5}$ \\
            \midrule
            \begin{filecontents*}{fig1_summary.csv}
method,alpha,mse,p2_5,p97_5,num_seeds,analysis_time_start,analysis_time_end,reference_4_Ny_r2
add,0.0,2443.8677887025538,2131.252830244316,2754.009790540524,20,501,1000,160.0
add,0.5,24.853347016969185,23.749364834513173,26.13877584568768,20,501,1000,160.0
add,2.0,116.4684023058132,115.08632831902437,118.14888943223238,20,501,1000,160.0
add-proj,0.0,2235.079897969959,1751.3651710494996,2680.6085205192685,20,501,1000,160.0
add-proj,0.5,25.237660589095547,24.330384294003412,26.168971325867357,20,501,1000,160.0
add-proj,2.0,119.90860586596686,118.49819102197927,121.77996568313804,20,501,1000,160.0
\end{filecontents*}
\csvreader[
                late after line=\\
            ]{fig1_summary.csv}{
                method=\methodName,
                alpha=\alphaValue,
                mse=\mseValue,
                p2_5=\lowerValue,
                p97_5=\upperValue
            }{%
                \methodName & $\alphaValue$
                & \num[round-mode=places,round-precision=2]{\mseValue}
                & \num[round-mode=places,round-precision=2]{\lowerValue}
                & \num[round-mode=places,round-precision=2]{\upperValue}
            }
            \bottomrule
        \end{tabular*}
    \end{table}
\end{revision}
As shown in~Figure~\ref{fig:se}, the PO method with the additive inflation (add) achieves an MSE comparable to that of the PO method with the projected additive inflation (add-proj).
Without inflation ($\alpha=0.0$), each MSE is significantly larger than the estimated bound indicated by the black line.
With strong inflation ($\alpha=2.0$), each MSE is smaller than the estimated bound. This implies that additive inflation effectively reduces the state estimation error of the PO method.
With small inflation ($\alpha=0.5$), each MSE is even smaller than that under strong inflation.
These results \rev{are consistent with} the error bounds derived in~Theorems~\ref{thm:error_bound} and \ref{thm:error_bound2}.
\rev{The estimated error bound $4\Ny r^2$ turns out to be a useful reference: the errors with inflation stay below it while remaining of a comparable order, although this estimate relies on the assumption that $\theta$ is small.
This supports the applicability of the error control by \eqref{eq:error_control} in practice.
}
\rev{The comparable accuracies between the projected and non-projected methods imply that additional information transfer in the unobserved subspace does not improve accuracy, which is consistent with the fact that the observed $2/3$-pattern is dominant in the dynamics of the Lorenz 96 model as noted in Remark~\ref{rem:l96_observation}.}
\rev{On the other hand, the results} suggest that our theoretical results do not necessarily yield the optimal value of the inflation parameter $\alpha$ in terms of minimizing the MSE.

We then show the spatio-temporal structure of the error.
Let $ \berr^{(k)}_n = (\berr^{(k),i}_n)_{i=1}^J $.
Figure~\ref{fig:error_st} shows the spatio-temporal plot of a sample path of absolute error $a(i,n) = |\berr^{(k),i}_n|$ of a member $k=1$ in the PO method with the additive inflation (add) and the projected additive inflation (add-proj) for $\alpha=0.0, 0.5, 2.0$.
Similar behaviors are observed between both inflation methods: the error is spatially distributed with high values when $\alpha=0.0$ and stays small when $\alpha=0.5, 2.0$.
Especially for $\alpha=0.5, 2.0$, in the initial period from $n=1$ to $n=200$, the error is localized around the unobserved components and decays rapidly.
The spatio-temporal plots of the associated true state and observation, and analysis member ($k=1$) are shown in Appendix~\ref{sec:appendix_stplot}.
\begin{figure}[pos=!ht]
    \centering
    \includegraphics[width=0.75\columnwidth]{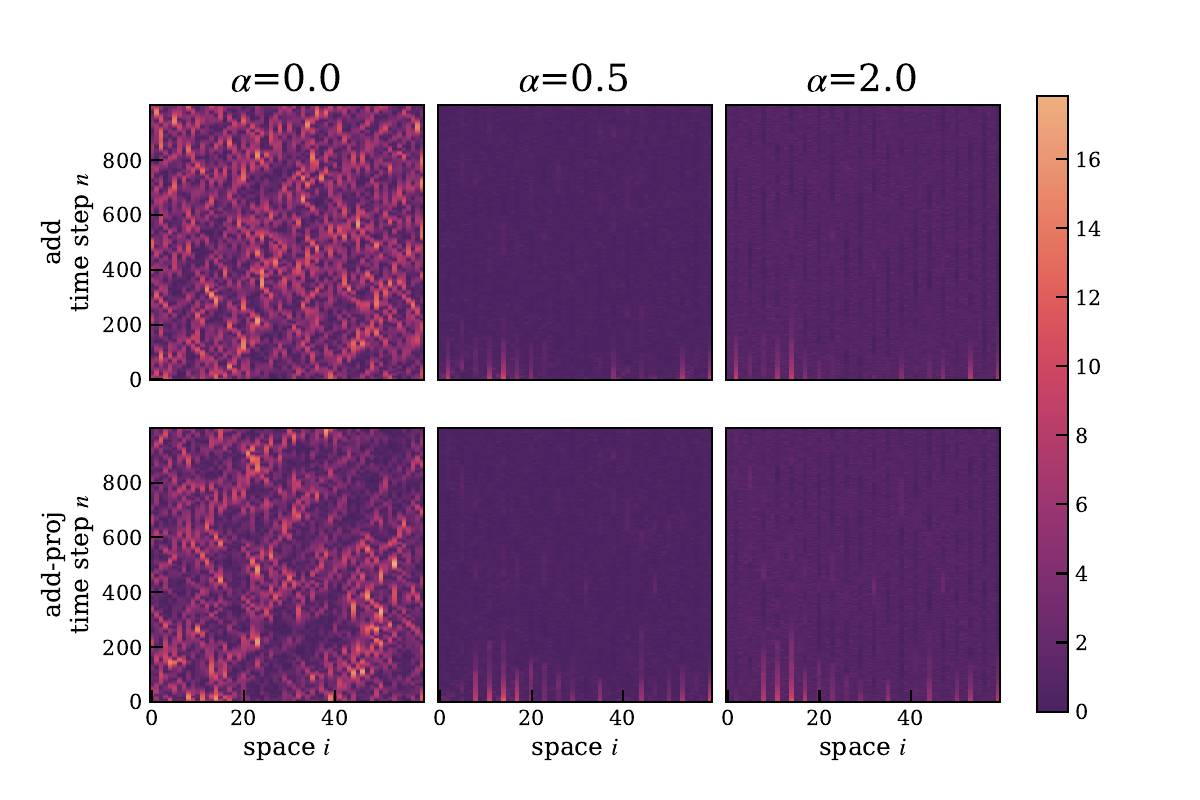}
    \caption{Spatio-temporal plot of a sample path of absolute error $a(i,n) = |\berr^{(k),i}_n|$ of a member $k=1$ in the PO method with the additive inflation (add) and the projected additive inflation (add-proj) for $\alpha=0.0, 0.5, 2.0$.}
    \label{fig:error_st}
\end{figure}
Figures~\ref{fig:se} and~\ref{fig:error_st} imply that even though the PO method with the additive inflation (add) corrects unobserved components in the analysis step, it does not outperform that with the projected additive inflation (add-proj).
This behavior reflects the dominance of the \rev{$2/3$-}observed components in the dynamics of the Lorenz 96 model as noted in Remark~\ref{rem:l96_observation}: the dynamics itself has the ability to determine the unobserved components up to the observation accuracy even with the noisy estimate of the dominant components.

We then compare the structure of the analysis covariance matrices $ \Pa_n $ between both inflation methods and their dependencies to $\alpha$.
Hence, we perform additional experiments with large values of $\alpha = 10.0, 100.0$.
Figure~\ref{fig:analysis_cov} shows the analysis covariance matrices of the PO method with the additive inflation (add) and the projected additive inflation (add-proj) for $ \alpha=0.0, \dots, 100.0$, where a realization of random noises is used and the covariance at $ n=N_t $ is used.
Each component is normalized by the maximum absolute value in all components and the indices $i$ and $j$ are rearranged so that the observed components are at the beginning.
We do not observe a typical structure in the analysis covariance matrices with $\alpha=0.0$.
On the other hand, for the additive inflation with $\alpha=0.5, 2.0$, the uncertainty in the analysis covariance is concentrated on the observed space.
\begin{figure}[pos=!ht]
    \centering
    \includegraphics[width=0.75\columnwidth]{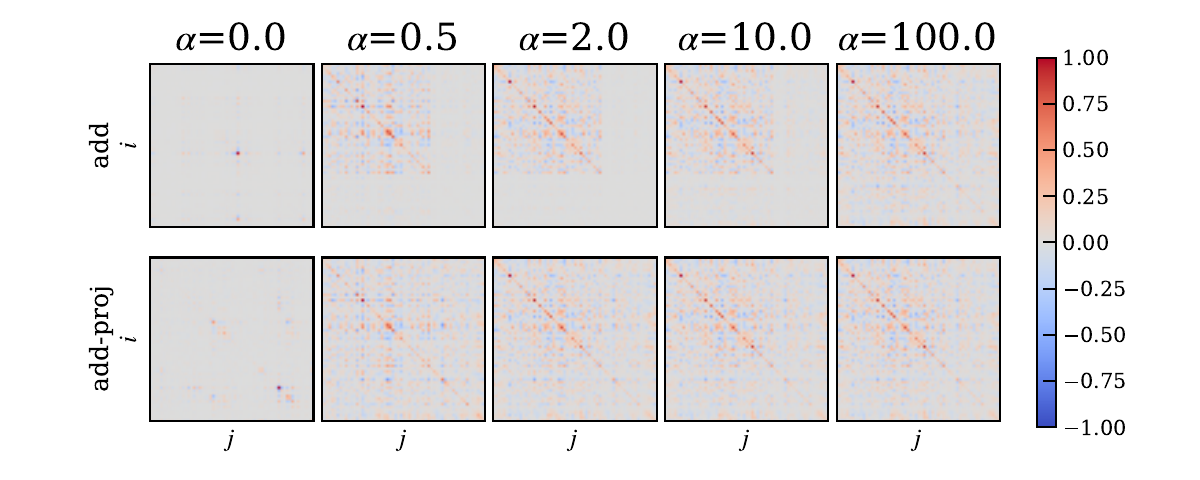}
    \caption{The normalized and rearranged analysis covariance matrices of the PO method with the additive inflation (add) and the projected additive inflation (add-proj) for $\alpha=0.0, \dots, 100.0$, where a realization of random noises is used and the covariance at $n=N_t$ is used for each panel.}
    \label{fig:analysis_cov}
\end{figure}
To quantify the concentration of the uncertainty on the observed space, we define the ratio of the off-diagonal block to the observed block in the analysis covariance $\Pa_n$ by
\begin{align*}
    \mathcal{R}_n = \frac{\fnorm{(I - \Proj) \Pa_n \Proj}}{\fnorm{\Proj\Pa_n\Proj}}.
\end{align*}
Figure~\ref{fig:off_diag_ratio} shows the time series of \rev{$\E[\mathcal{R}_n]$} of the PO method with the additive inflation (add) and the projected additive inflation (add-proj) for $\alpha=0.0, \dots, 100.0$.
The ratios for `add-proj' approach the ratio with $\alpha=100.0$ from above.
The latter is bounded from below by $0.25$.
On the other hand, some series for `add' ($\alpha=0.5, 2.0, 10.0$) quickly decay to a level smaller than $0.25$.
The ratio for `add' with $\alpha=100.0$ approaches from below to that for `add-proj' with $\alpha=100.0$, which is consistent with the convergent behavior of covariance matrices in Figure~\ref{fig:analysis_cov}.
\rev{We note that a small analysis covariance ratio for add does not imply the absence of information transfer; rather, it indicates that information transfer effectively corrects unstable components during the analysis step.}
\begin{figure}[pos=!ht]
    \centering
    \includegraphics[width=0.75\columnwidth]{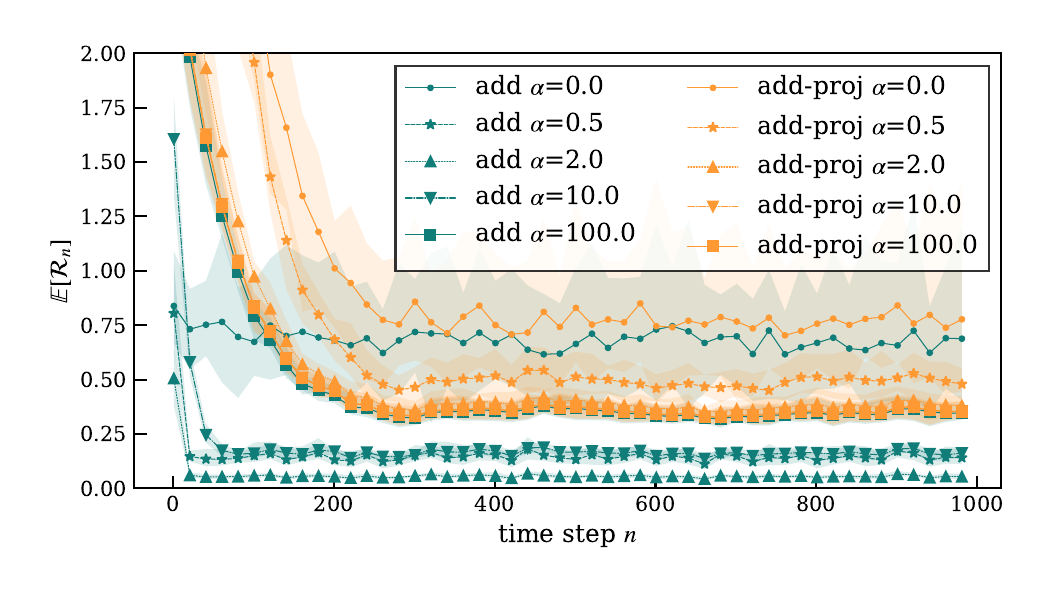}
    \caption{Time series of the ratio of the off-diagonal block to the observed block in the analysis covariance of the PO method with the additive inflation (add) and the projected additive inflation (add-proj) for $\alpha=0.0, \dots, 100.0$. \rev{The 95\% intervals are shown in shaded color.}}
    \label{fig:off_diag_ratio}
\end{figure}

\begin{revision}
    \subsection{Comparatively sparse observation (EXP2)}
    In this section, we use the comparatively sparse $1/2$-pattern observation matrix \eqref{eq:l96_observation_sparse} to compare the PO methods in a more practical situation.
    Since accuracy in approximation of the observed--unobserved correlations depends on an ensemble size $m$, we perform experiments for various $\Ne$ with a fixed $\alpha=0.5$.
    To measure the information transfer to the unobserved components, we consider the analysis increment on the unobserved subspace.
    \begin{align*}
        I_n^{\mathrm{U}} \coloneq \left(\frac{1}{\#\mathcal{I}^{\mathrm{U}}} \sum_{i \in \mathcal{I}^{\mathrm{U}}} \left|\left(\overline{\zf}_n\right)^i - \left(\overline{\za}_n\right)^i\right|^2\right)^{1/2},
    \end{align*}
    where $\mathcal{I}^{\mathrm{U}} = \{2, 4, \dots, 60\}$ is a set of indices of unobserved components associated with \eqref{eq:l96_observation_sparse} and the overline $\overline{{}*{}}$ represents an ensemble mean.
    Moreover, to compare accuracies for the two methods, we introduce the time and ensemble averaged MSE ratio as
    \begin{align*}
        \frac{\sum_{n=1}^{N_t}\sum_{k=1}^{\Ne} \left\|\berr_n^{(k), \mathrm{add}}\right\|^2}{\sum_{n=1}^{N_t}\sum_{k=1}^{\Ne} \left\|\berr_n^{(k), \mathrm{add-proj}}\right\|^2},
    \end{align*}
    where $\berr$ with the superscripts `add' and `add-proj' represent the state estimation error of corresponding methods respectively.
    Figure~\ref{fig:projection_comparison} shows $20$ random realizations of (a) the time averaged increment and the time and ensemble averaged MSE ratio with their sample means and empirical 95\% intervals.
    Table~\ref{tbl:projection_comparison} provides the quantitative summary corresponding to Figure~\ref{fig:projection_comparison}.
    In Figure~\ref{fig:projection_comparison}~(a), while the increment in `add-proj' is equal to zero by construction, that in `add' stays around non-zero values and exhibits low variability for $\Ne=80$.
    This implies that substantial updates in the unobserved space are applied in the analysis step of the PO method without covariance projection, i.e., observed--unobserved information transfer certainly exists.
    In Figure~\ref{fig:projection_comparison}~(b), all MSE ratios are less than $1.0$, implying that `add' achieves higher accuracy than `add-proj' in this setting.
    Such difference in filter performance is not observed under $2/3$-pattern observation because the unobserved errors decay under the dynamics itself.
    On the other hand, under the current $1/2$-pattern observation, unobserved errors do not decay under the dynamics and are effectively corrected in the analysis step of `add', i.e., observed--unobserved information transfer substantially works.
    Another study for an ensemble filter utilizing the EnKF analysis update also reports that analysis increment into the unobserved space based on observed--unobserved cross-correlations improves accuracy under 25\% observational coverage of the Lorenz 96 model \cite{xiongTwoStepEnsembleScore2026}.
    \begin{figure}[pos=!ht]
        \centering
        \includegraphics[width=0.75\columnwidth]{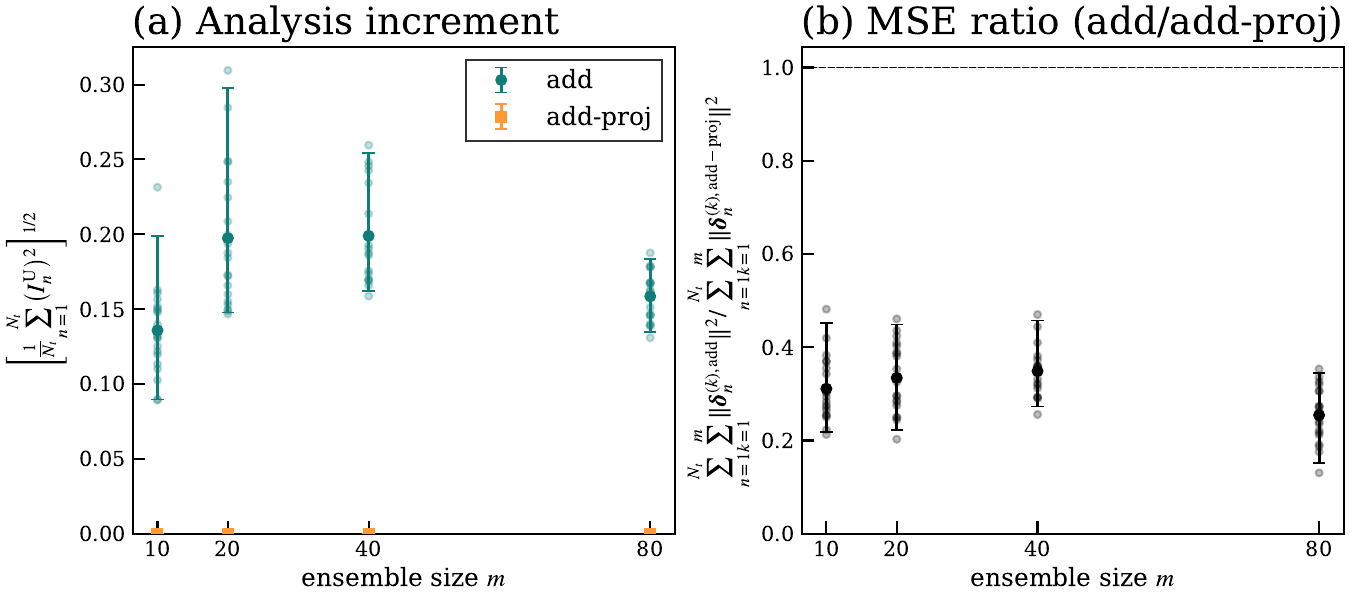}
        \caption{Scatter plots for $20$ random realizations of (a) the time averaged increment and (b) the time and ensemble averaged MSE ratio with their sample means and empirical 95\% intervals. The PO method with the additive inflation (add) and the projected additive inflation (add-proj) are used with parameters $\Ne=10, 20, 40, 80$ and $\alpha=0.5$.}
        \label{fig:projection_comparison}
    \end{figure}
    \begin{table}[pos=!ht]
        \caption{Quantitative summary corresponding to Figure~\ref{fig:projection_comparison}, for the comparatively sparse observation matrix~\eqref{eq:l96_observation_sparse} with $\alpha=0.5$.
        The values $p_{2.5}$ and $p_{97.5}$ are empirical percentiles. The time averaged increment for `add-proj' is omitted because it is equal to zero by construction.}
        \label{tbl:projection_comparison}
        \begin{tabular*}{\tblwidth}{@{}LCCCCCC@{}}
            \toprule
            & \multicolumn{3}{C}{Time averaged increment (add)} & \multicolumn{3}{C}{MSE ratio (add/add-proj)} \\
            \cmidrule(r){2-4}\cmidrule(l){5-7}
            $\Ne$ & Mean & $p_{2.5}$ & $p_{97.5}$ & Mean & $p_{2.5}$ & $p_{97.5}$ \\
            \midrule
            \begin{filecontents*}{fig5_summary.csv}
ensemble_size,alpha,increment_rms,increment_rms_p2_5,increment_rms_p97_5,mse_ratio,mse_ratio_p2_5,mse_ratio_p97_5,num_seeds,analysis_time_start,analysis_time_end
10,0.5,0.1358192536850716,0.08955135912477255,0.19885938428917616,0.3111105976834916,0.21768988116137142,0.4522440228574058,20,1,1000
20,0.5,0.19742412845366947,0.1478075024304246,0.2976308987824862,0.33434422249607726,0.22246684466938915,0.4490410052279196,20,1,1000
40,0.5,0.19896165565382667,0.16204176624775776,0.25419767015997186,0.34882353168929003,0.27250167482315474,0.4576389110055169,20,1,1000
80,0.5,0.15857193629296745,0.1347049215907852,0.18335818902863882,0.25435122342989824,0.15196845742701756,0.34494630966660195,20,1,1000
\end{filecontents*}
\csvreader[
                late after line=\\
            ]{fig5_summary.csv}{
                ensemble_size=\ensembleSize,
                increment_rms=\incrementValue,
                increment_rms_p2_5=\incrementLower,
                increment_rms_p97_5=\incrementUpper,
                mse_ratio=\ratioValue,
                mse_ratio_p2_5=\ratioLower,
                mse_ratio_p97_5=\ratioUpper
            }{%
                $\ensembleSize$
                & \num[round-mode=places,round-precision=3]{\incrementValue}
                & \num[round-mode=places,round-precision=3]{\incrementLower}
                & \num[round-mode=places,round-precision=3]{\incrementUpper}
                & \num[round-mode=places,round-precision=3]{\ratioValue}
                & \num[round-mode=places,round-precision=3]{\ratioLower}
                & \num[round-mode=places,round-precision=3]{\ratioUpper}
            }
            \bottomrule
        \end{tabular*}
    \end{table}
\end{revision}

\section{\rev{Conclusions}}
\label{sec:summary}
We have established uniform-in-time error bounds for the PO method for the partially observed Lorenz 96 model, both with and without the covariance projection.
To obtain the theoretical bounds, we \rev{use} additive covariance inflation \rev{and the particular $2/3$-pattern observation matrix \eqref{eq:l96_observation}}.
While the covariance projection forces the matrix product to be symmetric, our analysis without the projection successfully addresses the core difficulty of handling non-symmetric matrices that naturally emerge in the filtering process.
The \rev{first} numerical example \rev{(EXP1)} demonstrate\rev{s} comparable accuracies between the standard PO method and the projected PO method, with errors consistently remaining within the estimated bounds for sufficiently large inflation parameters.
A similar structure of the analysis covariance matrices is also observed between both inflation methods for \rev{a} large inflation parameter.
\rev{On the other hand, the second example (EXP2) with the $1/2$-pattern observation \eqref{eq:l96_observation_sparse} shows that the standard non-projected PO method provides a substantial magnitude of the analysis increment in the unobserved space and improvement in accuracy compared to the projected PO method. This provides numerical evidence that covariance projection increases the state estimation error under the comparatively sparse observations while the non-projected method retaining the cross-covariance leads to effective observed--unobserved information transfer.}

\rev{As in existing studies, Theorems~\ref{thm:error_bound} and \ref{thm:error_bound2} rely on the particular $2/3$-pattern observation matrix associated with the Lorenz 96 model, as the error in the unobserved space decays along with the time evolution of the system.}
\rev{Since the existence of such an appropriate observation operator is already known for typical dissipative dynamical systems \cite{sanz-alonsoLongTimeAsymptoticsFiltering2015}, it is straightforward to extend our analysis to such systems.}
\rev{On the other hand, it is hard to extend the current analysis for comparatively sparse observation matrices such as the $1/2$-pattern observation for the Lorenz 96 model; we need not only to handle non-symmetric structure but also to estimate how the ensemble covariance approximates cross-correlations between the observed and unobserved subspaces.}

Another important direction is to consider adaptive observation matrices that reflect the time-varying instability of the dynamics.
Even for 3DVar, error bounds have not yet been established for adaptive observation matrices.
To achieve this, we need to impose specific conditions on the dynamics that characterize its time-varying instability and theoretically connect them to the observation matrices.









\appendix

\section{Matrix calculations}
\label{sec:appendix_matrix}
\renewcommand{\theequation}{\thesection.\arabic{equation}}
\renewcommand{\thefigure}{\thesection.\arabic{figure}}
\renewcommand{\thetable}{\thesection.\arabic{table}}
\setcounter{equation}{0}
\setcounter{figure}{0}
\setcounter{table}{0}

We prepare a lemma for matrix calculations.
\begin{lemma}
    \label{lem:a1}
    Let $S, \Pi \succeq 0$. Then, we have
    \begin{align}
        \label{eq:lem:a11}
        \sigma(S\Pi) \subset [0, \infty),
    \end{align}
    where $\sigma({}\cdot{})$ denotes the spectrum of a matrix.
    Hence, $I + S\Pi$ is invertible.
    Moreover, if $\Pi$ is a projection matrix, we have
    \begin{align}
        \label{eq:lem:a12}
        \Pi(I + S\Pi)^{-1} = \Pi(I + \Pi S \Pi)^{-1} \Pi.
    \end{align}
\end{lemma}
\begin{proof}
    The spectrum bound~\eqref{eq:lem:a11} follows from~\cite{hladnikSpectrumProductOperators1988}.
    Since $\Pi^2 = \Pi$, we have
    \begin{align*}
        (I + \Pi S \Pi)\Pi = \Pi (I + S\Pi).
    \end{align*}
    Hence, it follows that
    \begin{align*}
        \Pi (I + S\Pi)^{-1} = (I + \Pi S \Pi)^{-1} \Pi.
    \end{align*}
    Left-multiplying by $\Pi$ yields~\eqref{eq:lem:a12}.
\end{proof}

\begin{lemma}
    \label{lem:a2}
    Let $\R^L = \mathcal{H}_1 \oplus \mathcal{H}_2$ be an orthogonal decomposition.
    We write a matrix $ A \in \R^{L \times L}$ in the block representation
    \begin{align*}
        A = \left(
        \begin{array}{cc}
            A_{11} & A_{12} \\
            A_{21} & A_{22}
        \end{array}
        \right),
    \end{align*}
    with respect to the decomposition.
    Then, we have
    \begin{align}
        \label{eq:lem:a2}
        \opnorm{A} \le \sqrt{\sum_{i,j=1}^2 \opnorm{A_{i,j}}^2}.
    \end{align}
\end{lemma}
\begin{proof}
    Let $\bx = \bx_1 \oplus \bx_2 \in \mathcal{H}_1 \oplus \mathcal{H}_2$.
    Then, we have
    \begin{align*}
        A\bx_1 \oplus \bx_2 = (A_{1, 1} \bx_1 + A_{1, 2} \bx_2) \oplus (A_{2, 1} \bx_1 + A_{2, 2} \bx_2).
    \end{align*}
    From the triangle inequality and the definition of an operator norm, it follows that
    \begin{align*}
        \norm{A\bx_1 \oplus \bx_2}^2 & = \sum_{i}^2 \norm{A_{i, 1} \bx_1 + A_{i, 2} \bx_2}^2 \le \sum_{i}^2 (\norm{A_{i, 1} \bx_1} + \norm{A_{i, 2} \bx_2})^2 \\
        & \le \sum_{i}^2 (\opnorm{A_{i, 1}} \norm{\bx_1} + \opnorm{A_{i, 2}}\norm{\bx_2})^2 \le \opnorm{A'}^2 (\norm{\bx_1}^2 + \norm{\bx_2}^2),
    \end{align*}
    where $ A' $ is a block matrix of operator norms:
    \begin{align*}
        A' = \left(
        \begin{array}{cc}
            \opnorm{A_{11}} & \opnorm{A_{12}} \\
            \opnorm{A_{21}} & \opnorm{A_{22}}
        \end{array}
        \right) \in \R^{2 \times 2}.
    \end{align*}
    Therefore, we have $\opnorm{A} \le \opnorm{A'}$.
    Taking the Frobenius norm of $A'$ yields~\eqref{eq:lem:a2}.
\end{proof}

\section{Proofs}
\label{sec:appendix_proofs}
\setcounter{equation}{0}
\setcounter{figure}{0}
\setcounter{table}{0}
\begin{proof}[Proof of~Lemma~\ref{lem:pred_bound}]
    The proof follows a similar approach to~\cite{lawFilterAccuracyLorenz2016}.
    For $ t \in [0, \tau) $, subtracting~\eqref{eq:l96_dissipative} with $ \x $ from one with $ \z $ yields
    \begin{align}
        \label{eq:l96_err}
        \frac{d\berr}{dt} + \berr + 2\B(\bx, \berr) + \B(\berr, \berr) = 0.
    \end{align}
    We first derive \eqref{eq:pred_bound_proj}.
    By taking the inner product of both sides with $ \Proj\berr$, we have
    \begin{align*}
        &\frac{1}{2} \frac{d}{dt}\norm{\Proj \berr}^2 + \norm{\Proj \berr}^2 + 2 \bracket{\B(\bx,\berr)}{\Proj\berr} + \bracket{\B(\berr,\berr)}{\Proj\berr} = 0.
    \end{align*}
    By using the inequality $ \norm{\B(\bx,\widetilde{\bx})} \le 2 \norm{\bx}\norm{\widetilde{\bx}}$ for any $ \bx, \widetilde{\bx} \in \R^J $ (see Proposition~4.3 in~\cite{lawFilterAccuracyLorenz2016}), we obtain
    \begin{align*}
        & \frac{1}{2} \frac{d}{dt}\norm{\Proj \berr}^2 + \norm{\Proj \berr}^2 \le 4 \norm{\x}\norm{\berr}\norm{\Proj\berr} + 2\norm{\berr}^2 \norm{\Proj\berr} \\
        & \le 4 \rho\norm{\berr}\norm{\Proj\berr} + 4\rho \norm{\berr} \norm{\Proj\berr} = 8 \rho\norm{\berr}\norm{\Proj\berr} \le 16\rho^2\norm{\berr}^2 + \norm{\Proj\berr}^2.
    \end{align*}
    Here, the second inequality follows from $ \bx, \bz \in B(\rho) $, and the third inequality follows from the elementary inequality $ ab \le \frac{1}{2}(a^2 + b^2) $ for $ a = 4\rho\norm{\berr}, b = \norm{\Proj\berr}$.
    Therefore, we have
    \begin{align*}
        \frac{d}{dt}\norm{\Proj\berr}^2 \le 32 \rho^2 \norm{\berr}^2 \le 32 \rho^2 \exp(2\beta t)\norm{\berr(0)}^2.
    \end{align*}
    The last inequality follows from~Lemma~\ref{lem:l96_chaos}.
    Integrating both sides from $ 0 $ to $ t $ gives
    \begin{align*}
        \norm{\Proj\berr}^2 \le \frac{16\rho^2}{\beta}(\exp(2\beta t) - 1) \norm{\berr(0)}^2 + \norm{\Proj\berr(0)}^2.
    \end{align*}
    Hence, $ a_1 $ is given by~\eqref{eq:a1}.
    To derive~\eqref{eq:pred_bound}, we need to utilize the property that there is a constant $ c > 0 $ such that $ \bracket{\B({\bx}, \bx)}{\widetilde{\x}} \le c \norm{\bx}\norm{\widetilde{\bx}}\norm{\Proj\bx}$ for any $ \bx, \widetilde{\bx} \in \R^J $ (Proposition~4.3 in~\cite{lawFilterAccuracyLorenz2016}).
    Taking the inner product of~\eqref{eq:l96_err} with $ \berr$ and using this property \rev{yields}
    \begin{align*}
        &\frac{1}{2}\frac{d}{dt}\norm{\berr}^2 + \norm{\berr}^2 \le c\norm{\x}\norm{\berr}\norm{\Proj\berr} \le c\rho\norm{\berr}\norm{\Proj\berr} \le \frac{\norm{\berr}^2}{2} + \frac{c^2\rho^2}{2}\norm{\Proj\berr}^2 \\
        & \le \frac{\norm{\berr}^2}{2} + \frac{c^2\rho^2}{2}\left(a_1(t)\norm{\berr(0)}^2 + \norm{\Proj \berr(0)}^2\right).
    \end{align*}
    Here, the second inequality follows from $ \bx \in B(\rho) $, and the third inequality follows from the elementary inequality $ ab \le \frac{1}{2}(a^2 + b^2) $ for $ a = \norm{\berr}, b = c\rho\norm{\Proj\berr}$.
    Hence, we have 
    \begin{align*}
        \frac{d}{dt}\norm{\berr}^2 + \norm{\berr}^2 \le {c^2\rho^2}\left(a_1(t)\norm{\berr(0)}^2 + \norm{\Proj \berr(0)}^2\right)
    \end{align*}
    and applying Gronwall's lemma yields
    \begin{align*}
        & \norm{\berr(t)}^2 \le \exp(-t)\norm{\berr(0)}^2 + c^2 \rho^2 \norm{\berr(0)}^2 \exp(-t) \int_0^t \exp(s) a_1(s)ds + c^2 \rho^2 (1 - \exp(-t))\norm{\Proj \berr(0)}^2 \\
        & = \left(\frac{16c^2\rho^4}{\beta} \left[\frac{1}{2\beta + 1}(\exp(2\beta t) - \exp(-t)) - (1 - \exp(-t))\right] + \exp(-t)\right) \norm{\berr(0)}^2 + c^2 \rho^2 (1 - \exp(-t))\norm{\Proj\berr(0)}^2.
    \end{align*}
    Therefore, $ b_1 $ and $ b_2 $ are given by~\eqref{eq:b1} and~\eqref{eq:b2}, respectively.
    To obtain~\eqref{eq:pred_contraction}, the inequalities for $ a_1 $ and $ b_2 $ are straightforward.
    By differentiating $ b_1(t) $, we have
    \begin{align*}
        & \left.\frac{d}{dt}b_1(t) \right|_{t=0} = \left.\frac{16c^2\rho^4}{\beta} \left[\frac{1}{2\beta + 1}(2\beta \exp(2\beta t) + \exp(-t)) - \exp(-t)\right] - \exp(-t) \right|_{t=0} \\
        & = \frac{16c^2\rho^4}{\beta} \left(\frac{2\beta+1}{2\beta + 1} -1\right) - 1 = -1 < 0.
    \end{align*}
    Since $ b_1(0) = 1$ and its derivative at $t=0$ is strictly negative, there exists $ t > 0 $ such that $ b_1(t) < 1$.
\end{proof}

\begin{proof}[Proof of~Lemma~\ref{lem:analysis_bound}]
    First, we rewrite the analysis update~\eqref{eq:po:update} using Proposition~3.2 in~\cite{kellyWellposednessAccuracyEnsemble2014} as
    \begin{align}
        \label{eq:po:update2}
        (I_{J} + \Pf_n \Hobs^\top \Robs^{-1} \Hobs) \bz_n^{(k)} = \bzf_n^{(k)} + \Pf_n \Hobs^\top \Robs^{-1} \by_n^{(k)}
    \end{align}
    for $ k = 1, \dots, \Ne$.
    From~\eqref{eq:po:update2} with $ \Pf_n = \Pf_n^\alpha $ and $ \Robs = r^2 I_{2J'}$, we have
    \begin{align*}
        (I_{J} + r^{-2} \Pf_n^\alpha) \bz_n^{(k)} = \bzf_n^{(k)} + r^{-2}\Pf_n^\alpha \by_n^{(k)}
    \end{align*}
    for $ k= 1, \dots, \Ne $.
    It is easy to see that
    \begin{align*}
         (I_{J} + r^{-2} \Pf_n^\alpha) \bx_n = \bx_n + r^{-2}\Pf_n^\alpha \Hobs \bx_n.
    \end{align*}
    Subtracting both sides yields
    \begin{align}
        \label{eq:lem2-err}
        (I_{J} + r^{-2} \Pf_n^\alpha) \berr_n^{(k)} = \berrhat_n^{(k)} + r^{-2}\Pf_n^\alpha(\bxi_n + \bxi_n^{(k)}),
    \end{align}
    where $ \berrhat_n^{(k)} \coloneqq \bzf_n^{(k)} - \bx_n $ for $ k = 1, \dots, \Ne $.
    Since $ I_{J} + r^{-2}\Pf_n^\alpha \succ 0 $, let $ \berr_n^{(k)} $ be divided as $ \berr_n = \bepsilon_1^{(k)} + \bepsilon_2^{(k)}$:
    \begin{align*}
        &\bepsilon_1^{(k)} = (I_{J} + r^{-2}\Pf_n^\alpha)^{-1} \berrhat_n^{(k)}, \\
        &\bepsilon_2^{(k)} = (I_{J} + r^{-2}\Pf_n^\alpha)^{-1} r^{-2}\Pf_n^\alpha (\bxi_n + \bxi_n^{(k)}).
    \end{align*}
    Since $ \opnorm{(I_{J} + r^{-2}\Pf_n^\alpha)^{-1}} \le 1$, we have
    \begin{align}
        \label{eq:lem2-1}
        \norm{\bepsilon_1^{(k)}} \le \norm{\berrhat_n^{(k)}}.
    \end{align}
    Owing to the symmetric property of $ r^{-2}\Pf_n^\alpha $, we have
    \begin{align*}
        \opnorm{(I_{J} + r^{-2} \Pf_n^\alpha)^{-1}r^{-2}\Pf_n^\alpha} \le 1
    \end{align*}
    and hence
    \begin{align}
        \norm{\bepsilon_2^{(k)}} \le \opnorm{(I_{J} + r^{-2} \Pf_n^\alpha)^{-1}r^{-2}\Pf_n^\alpha} \norm{(\bxi_n + \bxi_n^{(k)})} \le \norm{\bxi_n + \bxi_n^{(k)}}.         \label{eq:lem2-2}
    \end{align}
    From~\eqref{eq:lem2-1},~\eqref{eq:lem2-2}, and an argument using the conditional independence as in the proof of Theorem~2 in~\cite{takedaUniformErrorBounds2024}, we have
    \begin{align*}
        \E\left[\norm{\berr_n^{(k)}}^2\right]
        & \le \E\left[\norm{\berrhat_n^{(k)}}^2\right] + \E\left[\norm{\bxi_n + \bxi_n^{(k)}}^2\right] \le \E\left[\norm{\berrhat_n^{(k)}}^2\right] + 2 \rev{\Ny} r^2,
    \end{align*}
    which is \eqref{eq:analysis_bound}.
    
    By applying $ \Proj $ to~\eqref{eq:lem2-err}, we have
    \begin{align}
        \label{eq:analysis_proj_proof}
        (\Proj + r^{-2}\Pf_n^\alpha)\Proj\berr_n^{(k)} = \Proj\berrhat_n^{(k)} + r^{-2}\Pf_n^\alpha(\bxi_n + \bxi_n^{(k)}).
    \end{align}
    By definition of $ \Pf_n^\alpha$, 
    \begin{align*}
        \Proj + r^{-2}\Pf_n^\alpha
        & = \Proj + r^{-2}(\Pf_n + \alpha^2 \Proj) \succeq  \Proj + r^{-2} \alpha^2 \Proj = (1 + r^{-2}\alpha^2)\Proj.
    \end{align*}
    Hence, the operator norm of its inverse on the subspace $ \Proj(\R^{J})$ is bounded as
    \begin{align*}
        \opnormsub{(\Proj + r^{-2}\Pf_n^\alpha)^{-1}}{\Proj(\R^{J})} \le \frac{1}{1 + r^{-2}\alpha^2}.
    \end{align*}
    By an argument similar to that for $ \norm{\berr_n^{(k)}}$, it follows that
    \begin{align*}
        \E\left[\norm{\Proj \berr^{(k)}_n}^2\right] & \le \Theta \E\left[\norm{\Proj \berrhat^{(k)}_n}^2\right] + 2 \rev{\Ny} r^2,
    \end{align*}
    where $ \Theta = \Theta(\alpha) = \left(\frac{r^2}{r^2 + \alpha^2}\right)^2$.
\end{proof}

\begin{proof}[Proof of~Lemma~\ref{lem:analysis_bound2}]
    We note that $\Pf_n^\alpha = \Pf_n + \alpha^2 I_J$ here.
    Under the conditions of~Theorem~\ref{thm:error_bound2},~\eqref{eq:po:update2} is reduced to
    \begin{align*}
        (I_{J} + r^{-2} \Pf_n^\alpha \Pi) \bz_n^{(k)} = \bzf_n^{(k)} + r^{-2}\Pf_n^\alpha \Pi \by_n^{(k)}.
    \end{align*}
    As in the proof of~Lemma~\ref{lem:analysis_bound}, we have
    \begin{align*}
        (I_{J} + r^{-2} \Pf_n^\alpha \Pi) \berr_n^{(k)} = \berrhat_n^{(k)} + r^{-2}\Pf_n^\alpha \Pi \rev{(\bxi_n+\bxi_n^{(k)})}.
    \end{align*}
    Here, $I_{J} + r^{-2} \Pf_n^\alpha \Pi$ is not symmetric in general.
    It is still invertible due to~\eqref{eq:lem:a11} from~Lemma~\ref{lem:a1}.
    Hence, applying $(I_{J} + r^{-2} \Pf_n^\alpha \Pi)^{-1}$ yields 
    \begin{align}
        \label{eq:analysis_proof2}
        \berr_n^{(k)} 
        & = (I_{J} + r^{-2} \Pf_n^\alpha \Pi)^{-1} \berrhat_n^{(k)} + (I_{J} + r^{-2} \Pf_n^\alpha \Pi)^{-1} r^{-2}\Pf_n^\alpha \Pi \rev{(\bxi_n+\bxi_n^{(k)})}.
    \end{align}
    It is easy to obtain~\eqref{eq:analysis_bound_proj2} since applying $\Pi$ to \eqref{eq:analysis_proof2} and~\eqref{eq:lem:a12} in~Lemma~\ref{lem:a2} provides
    \begin{align*}
        \Pi \berr_n^{(k)} 
        & = (\Pi + r^{-2} \Pi \Pf_n^\alpha \Pi)^{-1} \Pi \berrhat_n^{(k)} + (\Pi + r^{-2} \Pi \Pf_n^\alpha \Pi)^{-1} r^{-2} \Pi \Pf_n^\alpha \Pi \rev{(\bxi_n+\bxi_n^{(k)})},
    \end{align*}
    which is equivalent to~\eqref{eq:analysis_proj_proof}.
    Therefore, we obtain~\eqref{eq:analysis_bound_proj2} in the same manner as~\eqref{eq:analysis_bound_proj}.
    
    On the other hand, to obtain~\eqref{eq:analysis_bound2}, it is sufficient to prove that
    \begin{align}
        \label{eq:analysis_norm11} \opnorm{(I_{J} + r^{-2} \Pf_n^\alpha \Pi)^{-1}} \le \widetilde{\Theta}^{1/2}, \\
        \label{eq:analysis_norm12} \opnorm{(I_{J} + r^{-2} \Pf_n^\alpha \Pi)^{-1} r^{-2}\Pf_n^\alpha \Pi} \le 1
    \end{align}
    for some $\widetilde{\Theta} = \widetilde{\Theta}(\alpha) \ge 1$ with $\widetilde{\Theta} \rightarrow 1 $ as $\alpha \rightarrow \infty$.
    We first introduce a block representation with respect to the orthogonal decomposition $\ran{\Pi} \oplus \ker{\Pi}$.
    Let $ S_n = r^{-2} \Pf_n^\alpha $ and
    \begin{align*}
        S_n = \left(
        \begin{array}{cc}
            S_{11} & S_{12} \\
            S_{21} & S_{22}
        \end{array}
        \right)
    \end{align*}
    on $\ran{\Pi} \oplus \ker{\Pi}$.
    Based on this, we have
    \begin{align*}
        I + S_n \Pi = \left(
        \begin{array}{cc}
            I + S_{11} & O \\
            S_{21} & I
        \end{array}
        \right),
    \end{align*}
    and its inverse is represented as
    \begin{align*}
        (I + S_n \Pi)^{-1} = \left(
        \begin{array}{cc}
            (I + S_{11})^{-1} & O \\
            -S_{21}(I + S_{11})^{-1} & I
        \end{array}
        \right).
    \end{align*}
    From~Lemma~\ref{lem:a2}, we have
    \begin{align*}
        \opnorm{(I + S_n \Pi)^{-1}}^2
        & \le \opnorm{(I + S_{11})^{-1}}^2 + \opnorm{S_{21}}^2\opnorm{(I + S_{11})^{-1}}^2 + 1.
    \end{align*}
    where $\opnorm{(I + S_{11})^{-1}} = \opnorm{\left[\Pi + \Pi r^{-2} (\Pf_n + \alpha^2 I)\Pi\right]^{-1}} \le (1 + r^{-2}\alpha^2)^{-1}$.
    Since $S_{21} $ is equivalent to $ \Pi(r^{-2}\Pf_n + \alpha^2 I) (I - \Pi) = \Pi(r^{-2}\Pf_n) (I - \Pi)$, it follows that $\opnorm{S_{21}} \le r^{-2} \opnorm{\Pf_n}$.
    Moreover, there exists a constant $c = c(\rho) > 0$ such that $\opnorm{\Pf_n} \le c$ since $\bzf_n^{k} \in B(\rho) $ for $k=1, \dots, m$.
    Therefore, the operator norm is estimated as
    \begin{align*}
        \opnorm{(I + S_n \Pi)^{-1}}^2 \le \Theta \coloneqq \frac{1 + r^{-4}c(\rho)^2}{(1 + r^{-2}\alpha^2)^{2}}  + 1 \rightarrow 1
    \end{align*}
    as $\alpha \rightarrow \infty$. This implies~\eqref{eq:analysis_norm11}.
    For~\eqref{eq:lem:a12}, we compute 
    \begin{align*}
        (I + S_n \Pi)^{-1}S_n \Pi 
        & = \left(
        \begin{array}{cc}
            (I + S_{11})^{-1} & O \\
            -S_{21}(I + S_{11})^{-1} & I
        \end{array}
        \right)
        \left(
        \begin{array}{cc}
            S_{11} & O \\
            S_{21} & O
        \end{array}
        \right) = \left(
        \begin{array}{cc}
            (I + S_{11})^{-1}S_{11} & O \\
            -S_{21}(I + S_{11})^{-1}S_{11} + S_{21} & O
        \end{array}
        \right) \\
        & = \left(
        \begin{array}{cc}
            (I + S_{11})^{-1}S_{11} & O \\
            S_{21}(I + S_{11})^{-1} & O
        \end{array}
        \right).
    \end{align*}
    Here, the last equality follows from $(I + S_{11})^{-1}S_{11} = I - (I + S_{11})^{-1}$.
    From~Lemma~\ref{lem:a2}, we have
    \begin{align*}
        \opnorm{(I + S_n \Pi)^{-1}S_n \Pi}^2 
        & \le \opnorm{(I + S_{11})^{-1}S_{11}}^2 + \opnorm{S_{21}(I + S_{11})^{-1}}^2.
    \end{align*}
    Due to symmetry, the first term in the upper bound is computed as
    \begin{align*}
        \opnorm{(I + S_{11})^{-1}S_{11}} = \lambda_{max}\left( (I + S_{11})^{-1}S_{11} \right) = \frac{\lambda_{max}(S_{11})}{1 + \lambda_{max}(S_{11})} = \frac{r^{-2}(\lambda_{max}(\Pf_n)+ \alpha^2)}{1 + r^{-2}(\lambda_{max}(\Pf_n) + \alpha^2)},
    \end{align*}
    where the second equality follows from the spectral mapping theorem and the monotonic property of the function $x \mapsto x/(1+x)$.
    The second term in the upper bound is estimated as
    \begin{align*}
        \opnorm{S_{21}(I + S_{11})^{-1}}
        & \le \opnorm{S_{21}}\opnorm{(I + S_{11})^{-1}} \le \frac{\opnorm{S_{21}}}{1 + r^{-2}\alpha^2}. 
    \end{align*}
    By putting $s=\frac{1}{1 + r^{-2}\alpha^2}$, $\mu = r^{-2}\lambda_{max}(\Pf_n)$, and $\nu=\opnorm{S_{21}} $, we have
    \begin{align*}
        \opnorm{(I + S_n \Pi)^{-1}S_n \Pi}^2 \le g(s) \coloneqq \left(\frac{(\mu-1)s + 1}{\mu s + 1}\right)^2 + \nu^2 s^2.
    \end{align*}
    Taking the derivative yields
    \begin{align*}
        \left. \frac{d}{ds} g(s) \, \right|_{s=0} = 2 \frac{(\mu-1)s + 1}{\mu s + 1} \frac{(\mu-1)(\mu s + 1) - ((\mu-1)s + 1)\mu}{(\mu s + 1)^2} + \left.2 \nu^2 s \,\right|_{s=0}  = -2 < 0.
    \end{align*}
    Since $g(0) = 1$, it follows that $g(s) < 1$ for sufficiently small $s>0$.
    This implies that the inequality~\eqref{eq:analysis_norm12} holds for sufficiently large $\alpha$.
\end{proof}

\section{Spatio-temporal plot of the states \rev{(EXP1)}}
\label{sec:appendix_stplot}
We show the spatio-temporal behavior of the states under the same conditions as those in Figure~\ref{fig:error_st}.
Figure~\ref{fig:true_obs_st} shows the true state and observation from time step $n=1$ to $n=1000$, where the white color represents the unobserved components.
Figure~\ref{fig:analysis_st} shows the analysis states of the PO method with the additive inflation (add) and the projected additive inflation (add-proj) for $\alpha=0.0, 0.5, 2.0$.
\begin{figure}[pos=!ht]
    \centering
    \includegraphics[width=0.75\columnwidth]{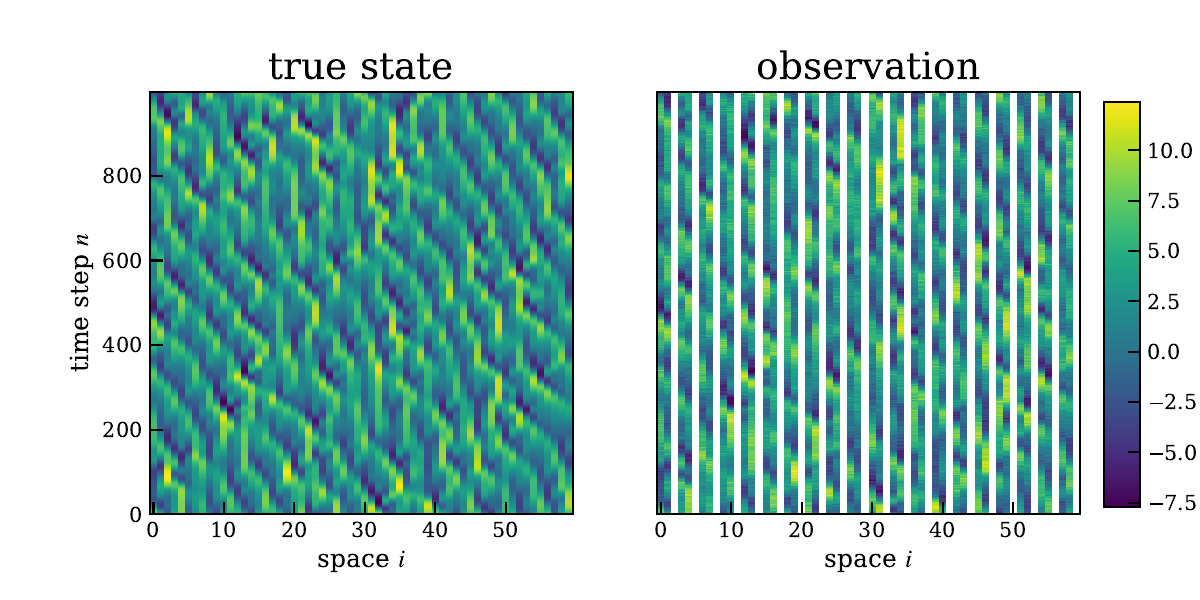}
    \caption{Spatio-temporal plots of the true state and observation of the Lorenz 96 model for $J=60$ and $F=8.0$ with the partial observation matrix \eqref{eq:l96_observation}. The white color represents the unobserved components. A realization of random noises is used for the observation.}
    \label{fig:true_obs_st}
\end{figure}

\begin{figure}[pos=!ht]
    \centering
    \includegraphics[width=0.75\columnwidth]{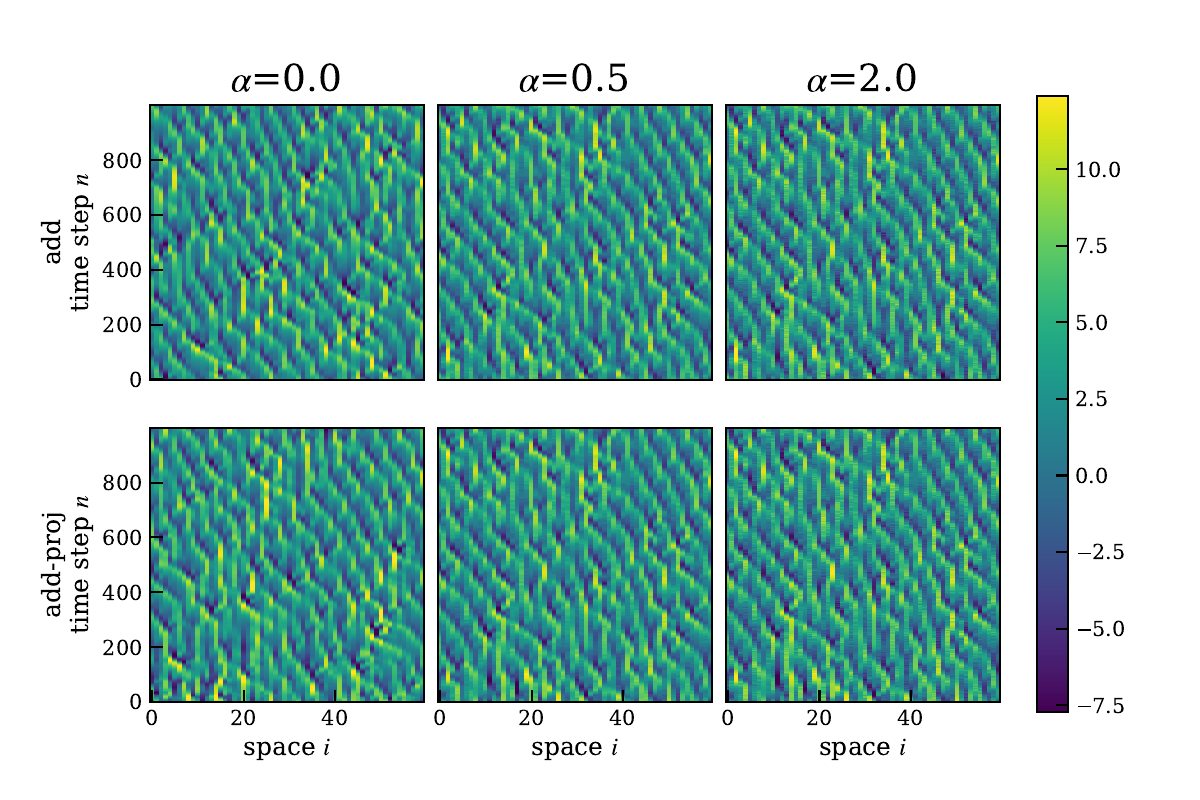}
    \caption{Spatio-temporal plots of a realization of analysis states $\za^{(k),i}_n$ of a member $k=1$ in the PO method with the additive inflation (add) and the projected additive inflation (add-proj) for $\alpha=0.0, 0.5, 2.0$.}
    \label{fig:analysis_st}
\end{figure}

\ifanonymized
\else
  \section*{Acknowledgments}
  I would like to thank Professor Takashi Sakajo for helpful discussions.
\fi

\printcredits

\ifanonymized
\else
  \section*{Data availability}
  The code is available at {https://github.com/KotaTakeda/l96\_po\_partialobs/releases/tag/\rev{v2.0.0}} and archived on Zenodo: {https://doi.org/10.5281/zenodo.\rev{21859188}}\rev{ (all versions: {https://doi.org/10.5281/zenodo.19605402})}.
\fi

\section*{\rev{Declaration of generative AI and AI-assisted technologies in the manuscript preparation process}}
\rev{During the preparation of this work, the author used Claude Code for language editing and for
drafting revisions to the text. The author reviewed and edited the output as needed and takes full
responsibility for the content of the published article.}

\bibliographystyle{elsarticle-num}

\bibliography{references}



\end{document}